\documentclass[a4paper,11pt,reqno]{amsart}
\usepackage{amsmath}
\usepackage[T1]{fontenc}
\usepackage{amssymb}
\usepackage{amsthm}
\usepackage{color}
\usepackage[lmargin=2.5 cm,rmargin=2.5 cm,tmargin=3.5cm,bmargin=2.5cm,
paper=a4paper]{geometry}

\newcommand{\Ab}{\mathbf A}
\newcommand{\Fb}{\mathbf F}

\newcommand{\R}{\mathbb R}
\newcommand{\Z}{\mathbb Z}
\newcommand{\N}{\mathbb N}
\newcommand{\C}{\mathbb C}
\DeclareMathOperator{\E0}{E_{\rm g}} 
\DeclareMathOperator{\IM}{Im}
\DeclareMathOperator{\curl}{curl}\DeclareMathOperator{\Div}{div}

\newtheorem{thm}{Theorem}[section]
\newtheorem{prop}[thm]{Proposition}
\newtheorem{lem}[thm]{Lemma}
\newtheorem{corol}[thm]{Corollary}

\theoremstyle{remark}
\newtheorem{rem}[thm]{Remark}

\numberwithin{equation}{section}
\title[2D Ginzburg-Landau  functional]{The ground state energy of the two dimensional Ginzburg-Landau functional with variable magnetic field}

\author[K.Attar]{}
\author[]{K. Attar}
\begin{document}
\begin{abstract}
We consider the Ginzburg-Landau functional with a variable applied
magnetic field in a bounded and smooth two
dimensional domain. We determine an accurate asymptotic formula for
the minimizing energy when the Ginzburg-Landau parameter and the magnetic field are large and  of the
same order. As a consequence, it is shown how bulk superconductivity decreases in
average  as the applied magnetic field increases.
\end{abstract}

\maketitle \textbf{AMS subject classifications.} 82D55
\section{Introduction}
\subsection{The functional and main results}
We consider a bounded open simply connected set $\Omega\subset\R^2$
with smooth boundary. We suppose that $\Omega$ models a
superconducting sample submitted to an applied external magnetic
field. The energy of the sample is given by the Ginzburg-Landau
functional,
\begin{multline}\label{eq-2D-GLf}
\mathcal E_{\kappa,H}(\psi,\Ab)= \int_\Omega\left[ |(\nabla-i\kappa
H\Ab)\psi|^2-\kappa^2|\psi|^2+\frac{\kappa^2}{2}|\psi|^4\right]\,dx
+\kappa^2H^2\int_{\Omega}|\curl\Ab-B_0|^2\,dx\,.
\end{multline}
Here $\kappa$ and $H$ are two positive parameters; $\kappa$ (the
Ginzburg-Landau constant) is a material parameter and $H$ measures
the intensity of the applied magnetic field. The wave function
(order parameter) $\psi\in H^1(\Omega;\C)$ describes the
superconducting properties of the material. The induced magnetic
field is $\curl \Ab$, where the potential
$\Ab\in{H}^1_{\Div}(\Omega)$, with $H^1_{\Div}(\Omega)$ is the space
defined in \eqref{eq-2D-hs} below. Finally, $B_{0}\in C^{\infty}(\overline{\Omega})$ is the
intensity of the external variable magnetic field and satisfies :
\begin{align}
|B_{0}| + |\nabla B_0 | >0 \mbox{ in } \overline{\Omega}\label{B(x)}\,.
\end{align}
The assumption in \eqref{B(x)} implies that for any open set $\omega$ relatively compact in $\Omega$  the set $\{x\in\omega, B_{0}(x)=0\}$ will be either empty, or consists of a union of smooth curves.
 Let $\Fb:\Omega\rightarrow \R^{2}$ be the vector field such that,
\begin{equation}\label{div-curlF}
\Div \Fb=0\,{\rm~and~}\,{\rm\curl \Fb}=B_{0}~{\rm
in~\Omega}\,,\,\,\,~\nu\cdot\Fb=0~{\rm on}~\partial\Omega.
\end{equation}
The vector $\nu$ is the unit interior normal vector of
$\partial\Omega$. The construction of $\Fb$ is recalled in the appendix. We define the
space,
\begin{equation}\label{eq-2D-hs} H^1_{\Div}(\Omega)=\{\Ab=(\Ab_{1},\Ab_{2})\in
H^1(\Omega)^{2}~:~\Div \Ab=0~{\rm in}~\Omega \,,\,\Ab\cdot\nu=0~{\rm
on}\,
\partial\Omega \,\}.
\end{equation}
Critical points $(\psi,\Ab)\in H^1(\Omega;\C)\times
H^1_{\Div}(\Omega)$ of $\mathcal E_{\kappa,H}$ are weak solutions of
the Ginzburg-Landau equations,
\begin{equation}\label{eq-2D-GLeq}
\left\{
\begin{array}{llll}
-(\nabla-i\kappa H\Ab)^2\psi=\kappa^2(1-|\psi|^2)\psi&{\rm in}&
\Omega
\\
-\nabla^{\bot}\curl(\Ab-\Fb)=\displaystyle\frac1{\kappa
H}\IM(\overline{\psi}\,(\nabla-i\kappa
H\Ab)\psi) &{\rm in}& \Omega\\
\nu\cdot(\nabla-i\kappa H\Ab)\psi=0&{\rm
on}&\partial\Omega\\
\curl\Ab=\curl\Fb&{\rm on}&\partial\Omega\,.
\end{array}\right.\end{equation}
Here, $\curl\Ab=\partial_{x_1}\Ab_{2}-\partial_{x_2}\Ab_{1}$ and
$\nabla^{\bot}\curl\Ab=(\partial_{x_2}(\curl\Ab),
-\partial_{x_1}(\curl\Ab)).$ If $\Div \Ab=0$, then
$\nabla^{\bot}\curl\Ab=\Delta \Ab$. In this paper, we study the
ground state energy defined as follows:
\begin{equation}\label{eq-2D-gs}
\E0(\kappa,H)=\inf\big\{ \mathcal E_{\kappa,H}(\psi,\Ab)~:~(\psi,\Ab)\in H^1(\Omega;\C)\times
H^1_{\Div}(\Omega)\big\}\,.
\end{equation}
More precisely, we give an asymptotic estimate which is valid in the simultaneous limit
$\kappa\rightarrow\infty$ and $H\rightarrow\infty$ in such a way that $\frac{H}{\kappa}$ remains asymptotically constant. The behavior of $\E0(\kappa,H)$ involves an auxiliary function
$g:[0,\infty)\to[-\frac12,0]$ introduced in \cite{SS02} whose
definition will be recalled in \eqref{eq-g(b)} below. The function
$g$ is increasing, continuous, $g(b)=0$ for all $b\geq 1$ and
$g(0)=-\frac{1}{2}$.
\begin{thm}\label{thm-2D-main}
Let $0<\Lambda_{\rm min}< \Lambda_{\rm max}$. Under Assumption \eqref{B(x)}, there exists positive
constants $C$, $\kappa_0$ and $\tau_0\in (1,2)$ such that if
\begin{align*}
\kappa_0 \leq \kappa,\qquad \Lambda_{\rm min}\leq
\frac{H}{\kappa}\leq  \Lambda_{\rm max}\,,
\end{align*}
then the ground state energy in \eqref{eq-2D-gs} satisfies,
\begin{equation}\label{eq-2D-thm}
 \left|\E0(\kappa,H)-\kappa^2\int_{\Omega}g\left(\displaystyle
\frac{H}{\kappa}\,|B_{0}(x)|\right)\,dx\right|\leq C\kappa^{\tau_0}\,.
\end{equation}
\end{thm}
Theorem \ref{thm-2D-main} was proved in \cite{SS02} when the magnetic field is constant $(B_{0}(x)=1)$. However, the estimate of the remainder is not explicitly given in \cite{SS02}.\\
The approach used in the proof of Theorem \ref{thm-2D-main} is slightly different from the one in \cite{SS02}, and is closer to that in \cite{FK} which studies the same problem when $\Omega\subset \R^{3}$ and $B_0$ constant.
\begin{corol}\label{corol-2D-main}
Suppose that the assumptions of Theorem~\ref{thm-2D-main} are satisfied.
Then the magnetic energy  of the minimizer  satisfies, for some positive constant C,
\begin{equation}\label{eq-2D-thm}
(\kappa H)^{2}\int_{\Omega}|\curl\Ab-B_{0}|^{2}\,dx \leq C\kappa^{\tau_{0}}\,.
\end{equation}
\end{corol}
\begin{rem}
The value of $\tau_0$ depends on the properties of $B_0$:  we find $\tau_0=\frac{7}{4}$ when $B_{0}$ does not vanish in $\overline{\Omega}$ and $\tau_0=\frac{15}{8}$ in the general case.
\end{rem}

\begin{thm}\label{thm-2D-op}
Suppose the assumptions of Theorem~\ref{thm-2D-main} are satisfied.
There exists positive constants $C$, $\kappa_0$ and a negative constant $\tau_{1}\in(-1,0)$ such that,
if $\kappa\geq\kappa_0\,,$ and $D\subset \Omega$ an open set, then the following is true.
\begin{enumerate}
\item If $(\psi,\Ab)\in H^1(\Omega;\C)\times H^1_{\Div}(\Omega)$ is a solution of \eqref{eq-2D-GLeq}, then,
\begin{equation}\label{eq-2D-op'}
\frac12\int_{D}|\psi|^4\,dx\leq
-\int_{D}g\left(\frac{H}{\kappa}|B_{0}(x)|\right)dx + C\kappa^{\tau_{1}}\,.
\end{equation}
\item
If $(\psi,\Ab)\in H^1(\Omega;\C)\times H^1_{\Div}(\Omega)$ is a
minimizer of \eqref{eq-2D-GLf}, then,
\begin{equation}\label{eq-2D-op}
\left|\int_{D}|\psi|^4\,dx+2\int_{D}g\left(\frac{H}{\kappa}|B_{0}(x)|\right)dx\right|\leq
C\kappa^{\tau_{1}}\,.
\end{equation}
\end{enumerate}
\end{thm}
\begin{rem}
The value of $\tau_1$ depends on the properties of $B_0$:  we find $\tau_1=-\frac{1}{4}$ when $B_{0}$ does not vanish in $\overline{\Omega}$ and $\tau_1=-\frac{1}{8}$ in the general case.
\end{rem}
\subsection {Discussion of main result:}
If $\{x\in\overline{\Omega}: B_{0}(x)=0\}\neq\varnothing$ and $H=b\kappa,\,\,\,\,b>0,$ then
$\displaystyle g\left(\frac{H}{\kappa}|B_{0}(x)|\right)\neq0$ in $D=\left\{x\in\Omega: \displaystyle \frac{H}{\kappa}|B_{0}(x)|<1\right\}$, and $|D|\neq 0$. Consequently, for $\kappa$ sufficiently large, the restriction of $\psi$ on $D$ is not zero in $L^{2}(\Omega)$.
This is a significant difference between our result and the one for constant magnetic field. When the magnetic field is a non-zero constant, then (see \cite{FH-b}), there is a universal constant $\circleddash_{0}\in (\frac{1}{2},1)$ such that, if $H=b\kappa$ and $b>\circleddash^{-1}_0$, then $\psi=0$ in $\overline{\Omega}$. Moreover, in the same situation, when $H=b\kappa$ and $1<b<\circleddash^{-1}_0$, then $\psi$ is small every where except in a thin tubular neighborhood of $\partial\Omega$ (see \cite{XP}).
Our result goes in the same spirit as in \cite{PK}, where the authors established under the Assumption~\eqref{B(x)} that when $H=b\kappa^{2}$ and $b>b_0$, then $\psi=0~\rm{in}~\overline{\Omega}$. ($b_{0}$ is a constant).
\subsection {Notation.}
{Throughout the paper, we use the following notation:}
\begin{itemize}
\item We write $\mathcal E$ for the functional $\mathcal E_{\kappa,H}$ in \eqref{eq-2D-GLf}.
\item The letter $C$ denotes a positive constant that is independent of the parameters $\kappa$ and $H$, and whose value  may change from a formula to another.
\item If $a(\kappa)$ and $b(\kappa)$ are two positive functions, we write $a(\kappa)\ll b(\kappa)$ if $a(\kappa)/b(\kappa)\to0$ as $\kappa\to\infty$.
\item If $a(\kappa)$ and $b(\kappa)$ are two functions with $b(\kappa)\not=0$, we write $a(\kappa)\sim b(\kappa)$ if $a(\kappa)/b(\kappa)\to1$ as $\kappa\to\infty$.
\item If $a(\kappa)$ and $b(\kappa)$ are two positive functions, we write $a(\kappa)\approx b(\kappa)$
if there exist positive constants $c_1$, $c_2$ and $\kappa_0$ such
that $c_1b(\kappa)\leq a(\kappa)\leq c_2b(\kappa)$ for all
$\kappa\geq\kappa_0$.
\item If $x\in\R$, we let $[x]_+=\max(x,0)$.
\item Given $R>0$ and $x=(x_{1},x_{2})\in\R^{2}$, we denote by
$Q_{R}(x)=(-R/2+x_{1},R/2+x_{1})\times(-R/2+x_{2},R/2+x_{2})$ the
square of side length $R$ centered at $x$.
\item We will use the standard Sobolev spaces $W^{s,p}$. For integer
values of $s$ these are given by
$$W^{n,p}(\Omega):=\big\{u\in L^{p}(\Omega):D^{\alpha}u\in L^{p}(\Omega)~{\rm for~all~}|\alpha|\leq n\big\}\,.$$
\item Finally we use the standard symbol $H^{n}(\Omega)=W^{n,2}(\Omega).$
\end{itemize}
\section{The limiting energy}\label{sec-LE}
\subsection{Two-dimensional limiting energy}\label{sec-LE2D}
Given a constant $b\geq 0$ and an open set $\mathcal D\subset \R^2$,
we define the following Ginzburg-Landau energy,
\begin{equation}\label{eq-LF-2D}
G^{\sigma}_{b,\mathcal D}(u)=\int_{\mathcal D}\left(b|(\nabla-i\sigma\Ab_0)u|^2
-|u|^2+\frac1{2}|u|^4\right)\,dx\,,\qquad\forall\,u\in H^{1}_{0}(D).
\end{equation}
Here $\sigma\in\{-1,+1\}$ and $\Ab_0$ is the canonical magnetic potential,
\begin{equation}\label{eq-hc2-mpA0}
\Ab_{0}(x)=\frac1{2}(-x_2,x_1)\,,\qquad\forall\,x=(x_1,x_2)\in\R^{2},
\end{equation}
that satisfies:
$$\curl\Ab_0=1\,\,{\rm~in~}\R^{2}\,.$$
We write $Q_{R}=Q_{R}(0)$ and let
\begin{equation}\label{m_0}
m_{0}(b,R)=\inf_{u\in H^1_0(Q_R;\C)} G^{+1}_{b,Q_R}(u)\,.
\end{equation}
\begin{rem}\label{m_0bar}
As $G^{+1}_{b,\mathcal D}(u)=G^{-1}_{b,\mathcal D}(\overline{u})$, it is immediate that,
\begin{equation}
\inf_{u\in H^1_0(Q_R;\C)} G^{-1}_{b,Q_R}(u)=\inf_{u\in H^1_0(Q_R;\C)} G^{+1}_{b,Q_R}(u)\,.
\end{equation}
\end{rem}
The main part of the next theorem was obtained by
Sandier-Serfaty \cite{SS02} and Aftalion-Serfaty \cite[Lemma
2.4]{AS}. However, the estimate in \eqref{eq-remainder} is obtained
by Fournais-Kachmar \cite{AK}.
\begin{thm}\label{thm-thmd-AS}
Let $m_0(b,R)$ be as defined in \eqref{m_0}.
\begin{enumerate}
\item For all $b\geq 1$ and $R>0$, we have $m_0(b,R)=0$.
\item For any $b\in[0,\infty)$, there exists a constant $g(b)\leq 0$ such that,
\begin{equation}\label{eq-g(b)}
 g(b)=\lim_{R\to\infty}\frac{m_0(b,R)} {|Q_R|}\,\,\,\,\,\,\,~{\rm
and~}\quad g(0)=-\frac{1}{2}\,.
\end{equation}
\item The function $[0,+\infty)\ni b\mapsto g(b)$ is continuous,
non-decreasing, concave and its range is the interval
$[-\frac12,0]$.
\item There exists a  constant $\alpha\in(0,\frac12)$ such
  that,
\begin{equation}\label{eq-prop-g(b)}\forall~b\in[0,1]\,,\quad
  \alpha(b-1)^2\leq|g(b)|\leq \frac12(b-1)^2\,.
\end{equation}
\item There exist  constants $C$ and $R_0$ such that,
\begin{equation}\label{eq-remainder}
\forall~R\geq R_0\,,~\forall~b\in[0,1]\,,\quad g(b)\leq
\frac{m_0(b,R)}{R^2}\leq g(b)+\frac{C}{R}\,.
\end{equation}
\end{enumerate}
\end{thm}
\section{A priori estimates}\label{sec-apest}

The aim of this section is to give {\it a priori} estimates for
solutions of the Ginzburg-Landau equations \eqref{eq-2D-GLeq}. These
estimates play an essential role in controlling the errors resulting
from various approximations.
The starting point is the following  $L^\infty$-bound  resulting
from the maximum principle. Actually, if $(\psi,\Ab)\in
H^1(\Omega;\C)\times H^1_{\Div}(\Omega)$ is a solution of
\eqref{eq-2D-GLeq}, then
\begin{equation}\label{eq-psi<1}
\|\psi\|_{L^\infty(\Omega)}\leq1\,.
\end{equation}
The set of estimates below is proved in \cite[Theorem~3.3 and Eq.
3.35]{FH-p} (see also \cite{Pa} for an earlier version).
\begin{thm}\label{thm-2D-apriori}~
Let $\Omega\subset\R^2$ be bounded and smooth and $B_{0}\in
C^{\infty}(\overline{\Omega})$.
\begin{enumerate}
\item
For all $p\in (1,\infty)$ there exists $C_p>0$ such that, if
$(\psi,\Ab)\in H^1(\Omega,C)\times H^1_{\Div}(\Omega)$ is a solution
of \eqref{eq-2D-GLeq}, then
\begin{equation}\label{eq-2D-ariori}
\|\curl \Ab-B_{0}\|_{W^{1,p}(\Omega)}\leq C_p\frac{1+\kappa
H+\kappa^2}{\kappa
  H}
\|\psi\|_{L^\infty(\Omega)}\|\psi\|_{L^2(\Omega)}\,.
\end{equation}
\item
For all $\alpha\in(0,1)$ there exists $C_\alpha>0$ such that, if
$(\psi,\Ab)\in H^1(\Omega,C)\times H^1_{\Div}(\Omega)$ is a solution
of \eqref{eq-2D-GLeq}, then
\begin{equation}\label{eq-2D-ariori'}
\|\curl \Ab-B_{0}\|_{C^{0,\alpha}(\overline{\Omega})}\leq
C_\alpha\frac{1+\kappa H+\kappa^2}{\kappa
  H}
\|\psi\|_{L^\infty(\Omega)}\|\psi\|_{L^2(\Omega)}\,.
\end{equation}
\item For all $p\in[2,\infty)$ there exists $C>0$ such that, if $\kappa>0$, $H>0$
  and $(\psi,\Ab)\in H^1(\Omega,C)\times
H^1_{\Div}(\Omega)$ is a solution of \eqref{eq-2D-GLeq}, then
\begin{equation}
\|(\nabla-i\kappa H\Ab)^2\psi\|_p\leq \kappa^2\|\psi\|_p\,,
\end{equation}
\begin{equation}\label{eq-2D-ariori''}
\|(\nabla-i\kappa H\Ab)\psi\|_2\leq \kappa\|\psi\|_2\,,
\end{equation}
\begin{equation}\label{eq-2D-ariori'''}
\|\curl(\Ab-\Fb)\|_{W^{1,p}(\Omega)}\leq \frac{C}{\kappa
H}\|\psi\|_\infty\|(\nabla-i\kappa H\Ab)\psi\|_p\,.
\end{equation}
\end{enumerate}
\end{thm}
\begin{rem}:
\begin{enumerate}

\item Using the $W^{k,p}$-regularity of the Curl-Div system \cite[Appendix A, Proposition~A.5.1]{FH-b}, we obtain from \eqref{eq-2D-ariori},
\begin{equation}\label{eq-2D-ariori1}
\| \Ab-\Fb\|_{W^{2,p}(\Omega)}\leq C_p\frac{1+\kappa
H+\kappa^2}{\kappa
  H}
\|\psi\|_{L^\infty(\Omega)}\|\psi\|_{L^2(\Omega)}\,.
\end{equation}
The estimate is true for any $p\in[2,\infty)$.
\item
Using the Sobolev embedding Theorem we
get, for all $\alpha\in(0,1)$
\begin{equation}\label{eq-2D-ariori1'}
\|\Ab-\Fb\|_{C^{1,\alpha}(\overline{\Omega})}\leq
C_\alpha\frac{1+\kappa H+\kappa^2}{\kappa
  H}
\|\psi\|_{L^\infty(\Omega)}\|\psi\|_{L^2(\Omega)}\,.
\end{equation}
\item
Combining \eqref{eq-2D-ariori''} and \eqref{eq-2D-ariori'''} (with
$p=2$) yields
\begin{equation}\label{eq-2D-apriori2}
\|\curl(\Ab-\Fb)\|_{L^2(\Omega)}
\leq\frac{C}{H}\|\psi\|_{L^\infty(\Omega)}\|\psi\|_{L^2(\Omega)}\,.\end{equation}
\end{enumerate}
\end{rem}
Theorem~\ref{thm-2D-apriori} is needed in order to obtain the
improved {\it a priori} estimates of the next theorem. Similar
estimates are given in \cite{Pa}.
\begin{thm}\label{thm-2D-ad-est}
Suppose that $0<\Lambda_{\min}\leq \Lambda_{\max}$. There exist
constants $\kappa_0>1$, $C_1>0$ and for any $\alpha\in(0,1)$,
$C_\alpha>0$ such that, if
\begin{equation}\label{kappaH}
\kappa\geq\kappa_0\,,\quad \Lambda_{\min}\leq\frac{H}{\kappa}\leq \Lambda_{\max}\,,
\end{equation}
and $(\psi,\Ab)\in H^1(\Omega;\C)\times H^1_{\Div}(\Omega)$ is a
solution of \eqref{eq-2D-GLeq}, then
\begin{align}
&\|(\nabla-i\kappa H\Ab)\psi\|_{C(\overline{\Omega})}
\leq C_1\sqrt{\kappa H}\|\psi\|_{L^\infty(\Omega)}\,,\label{eq-est1}\\
& \|\Ab-\Fb\|_{H^{2}(\Omega)}\leq
C_1\left(\|\curl(\Ab-\Fb)\|_{L^2(\Omega)}+\frac1{\sqrt{\kappa H}}
\|\psi\|_{L^2(\Omega)}\|\psi\|_{L^\infty(\Omega)}
\right),\label{eq-est2'}\\
&\|\Ab-\Fb\|_{C^{0, \alpha}(\overline{\Omega})}\leq C_\alpha
\left(\|\curl(\Ab-\Fb)\|_{L^2(\Omega)}+\frac1{\sqrt{\kappa
H}}\|\psi\|_{L^2(\Omega)}\|\psi\|_{L^\infty(\Omega)}\right)\,.
\label{eq-est2}
\end{align}
\end{thm}
\begin{proof}~\\
{\bf Proof of \eqref{eq-est1}:} See \cite[Proposition~12.4.4]{FH-b}.\\
{\bf Proof of \eqref{eq-est2'}:}\\
Let $a=\Ab-\Fb$. Since $\Div a=0$ and $a\cdot\nu=0$ on
$\partial\Omega$, we get by regularity of the curl-div system (see appendix, Proposition \ref{curl-div-reg}),
\begin{equation}
\|a\|_{H^{2}(\Omega)}\leq C\|\curl
a\|_{H^{1}(\Omega)}\,.\label{eq-est3}
\end{equation}
The second equation in \eqref{eq-2D-GLeq} reads as follows,
$$-\nabla^{\bot}\curl a=\frac1{\kappa H}\IM(\overline\psi\,(\nabla-i\kappa H\Ab)\psi)\,.$$
The estimates in \eqref{eq-est1}  and \eqref{eq-est3} now give,
$$\|a\|_{H^{2}(\Omega)}\leq C\left(\|\curl a\|_{L^2(\Omega)}+\frac{1}{\sqrt{\kappa H}}\,\|\psi\|_{L^2(\Omega)}\|\psi\|_{L^{\infty}(\Omega)}\right)\,.$$
{\bf Proof of \eqref{eq-est2}:}\\
This is a consequence  of the $\mathrm{S}$obolev embedding of
$H^{2}(\Omega)$ into $C^{0,\alpha}(\overline{\Omega})$ for any $\alpha\in (0,1)$ and
\eqref{eq-est2'}.
\end{proof} 
\section{Energy estimates in small squares}\label{sec-locest}
If $(\psi,\Ab)\in H^1(\Omega;\C)\times H^1_{\Div}(\Omega)$, we
introduce the energy density,
$$e(\psi,\Ab)=|(\nabla -i\kappa
H\Ab)\psi|^2-\kappa^2|\psi|^2+\frac{\kappa^2}2|\psi|^4\,.$$ We also
introduce the local energy of $(\psi,\Ab)$ in
a domain $\mathcal D\subset\Omega$ :
\begin{equation}\label{eq-GLe0}
\mathcal E_0(u,\Ab;\mathcal D)=\int_{\mathcal D}e(\psi,\Ab)\,dx\,.
\end{equation}
Furthermore, we define the Ginzburg-Landau energy of $(\psi,\Ab)$ in
a domain $\mathcal D\subset\Omega$ as follows,
\begin{equation}\label{eq-GLen-D}
\mathcal E(\psi,\Ab;\mathcal D)=\mathcal E_0(\psi,\Ab;\mathcal
D)+(\kappa H)^2\int_{\Omega}|\curl(\Ab-\Fb)|^2\,dx\,.
\end{equation}
If $\mathcal D=\Omega$, we sometimes omit the dependence on the
domain and write $\mathcal E_0(\psi,\Ab)$ for $\mathcal
E_0(\psi,\Ab;\Omega)$.
We start with a lemma that will be useful in the proof of
Proposition~\ref{prop-lb} below. Before we start to state the lemma, we
define for all $(\ell,x_0)$ such that $
\overline{Q_{\ell}(x_0)}\subset\Omega$,
\begin{equation}\label{sup_B0}
\overline{B}_{Q_{\ell}(x_0)}=\sup_{x\in Q_{\ell}(x_0)}|B_{0}(x)|\,,
\end{equation}
where $B_{0}$ is introduced in \eqref{B(x)}. Later $x_0$ will be chosen in a lattice of $\R^{2}$.
\begin{lem}\label{lem-lb}
For any $\alpha\in(0,1)$. there exist positive
constants $C$ and $\kappa_0$ such that if \eqref{kappaH} holds, $
0<\delta<1,\quad0<\ell<1$,
and $(\psi,\Ab)\in H^1(\Omega;\C)\times H^1_{\Div}(\Omega)$ is a
critical point of \eqref{eq-2D-GLf} $($i.e. a solution of
\eqref{eq-2D-GLeq}$)$, then, for any square $Q_{\ell}(x_0)$ relatively compact in $\Omega\cap\{|B_0|>0\}$, there exists $\varphi\in
H^{1}(\Omega)$, such that,
\begin{multline}
\mathcal E_0(\psi,\Ab;Q_{\ell}(x_0))\\
\quad \geq(1-\delta)\mathcal E_{0}(e^{-i\kappa H\varphi}\psi,\sigma_{\ell}\overline{B}_{Q_{\ell}(x_0)}\Ab_{0}(x-x_{0}),Q_{\ell}(x_0))-C\kappa^{2}\left(\delta^{-1}\ell^{2\alpha}+\delta^{-1}\ell^{4}\kappa^{2}+\delta\right)\int_{Q_{\ell}(x_0)}|\psi|^{2}\,dx\,,
\end{multline}
where $\sigma_{\ell}$ denotes the sign of $B_{0}$ in $Q_{\ell}(x_{0})$.
\end{lem}
\textit{Proof.}
\paragraph{\textbf{Construction of $\varphi$:}}~\\
 Let $\phi_{x_0}(x)=\big(\Ab(x_0)-\Fb(x_0)\big)\cdot x$, where $\Fb$
is the magnetic potential introduced in \eqref{div-curlF}. Using the
estimate in \eqref{eq-est2}, we get for all $x\in Q_{\ell}(x_0)$
and $\alpha\in(0,1)\,,$
\begin{align}\label{A-F}
|\Ab(x)-\nabla\phi_{x_0}-\Fb(x)|&=|(\Ab-\Fb)(x)-(\Ab-\Fb)(x_0)|\nonumber\\
&\leq\|\Ab-\Fb\|_{C^{0,\alpha}}\cdot|x-x_0|^{\alpha}\nonumber\\
&\leq C\frac{\sqrt{\lambda}}{\kappa H}\,\ell^{\alpha}\,,
\end{align}
where
$$\lambda=(\kappa H)^2\left(\|\curl(\Ab-\Fb)\|_{L^2(\Omega)}^2+\frac1{\kappa H}\|\psi\|^2_{L^2(\Omega)}\right)\,.$$
Using the bound $\|\psi\|_\infty\leq 1$ and the estimate in
\eqref{eq-2D-apriori2}, we get
\begin{equation}
\lambda\leq C\kappa^2\,,
\end{equation}
which implies that
\begin{equation}\label{eq-2D-est-g}
|\Ab(x)-\nabla\phi_{x_0}(x)-\Fb(x)|\leq C\,\frac{\ell^{\alpha}}{H}\,.
\end{equation}
We estimate the energy $\mathcal E_0(\psi,\Ab;Q_{\ell}(x_0))$ from below.
We will need the function $\varphi_{0}$
introduced in Lemma~\ref{app F} and satisfiying
$$|\Fb(x)-\sigma_{\ell}\overline{B}_{Q_{\ell}(x_0)}\Ab_{0}(x-x_{0})-\nabla\varphi_{0}(x)|\leq
C\ell^{2}\qquad{\rm in}~Q_{\ell}(x_0).$$ Let
\begin{equation}\label{defu}
u=e^{-i\kappa H\varphi}\psi\,,
\end{equation}
where $\varphi=\varphi_{0}+\phi_{x_{0}}$.
\paragraph{\textbf{Lower bound:}}~\\
 We start with estimating the kinetic energy from below as follows. For any
$\delta\in(0,1)$ and $\alpha\in(0,1)$, we write
\begin{align*}
|(\nabla-&i\kappa H\Ab)\psi|^2\nonumber\\
&=\Big|\Big(\nabla-i\kappa H(\sigma_{\ell}\overline{B}_{Q_{\ell}(x_0)}\Ab_{0}(x-x_{0})+\nabla\varphi)\Big)\psi-i\kappa H\Big(\Ab-\sigma_{\ell}\overline{B}_{Q_{\ell}(x_0)}\Ab_{0}(x-x_{0})-\nabla\varphi\Big)\psi\Big|^2\\
&\geq(1-\delta)\Big|\Big(\nabla-i\kappa H(\sigma_{\ell}\overline{B}_{Q_{\ell}(x_0)}\Ab_{0}(x-x_{0})+\nabla\varphi)\Big)\psi\Big|^2\\
&\qquad\qquad+(1-\delta^{-1})(\kappa
H)^2\Big|(\Ab-\nabla\phi_{x_{0}}-\Fb)\psi+(\Fb-\sigma_{\ell}\overline{B}_{Q_{\ell}(x_0)}\Ab_{0}(x-x_{0})-\nabla\varphi_{0})\psi\Big|^2\,.
\end{align*}
Using the estimates in \eqref{eq-2D-est-g}, \eqref{lem-F} and the assumptions in \eqref{kappaH}, we get,
$$|(\nabla-i\kappa H\Ab)\psi|^2\geq(1-\delta)\Big|\Big(\nabla-i\kappa
H(\sigma_{\ell}\overline{B}_{Q_{\ell}(x_0)}\Ab_{0}(x-x_{0})+\nabla\varphi)\Big)\psi\Big|^2-C\kappa^{2}\left(\delta^{-\frac1{2}}\ell^{2}H+\delta^{-\frac1{2}}\ell^{\alpha}\right)^{2}|\psi|^{2}\,.$$
Remembering the defintion of $u$ in \eqref{defu}, then, we deduce the lower bound of $\mathcal E_0$,
\begin{align}\label{es of A}
\mathcal E_0(\psi,&\Ab;Q_{\ell}(x_0))\nonumber\\
&\geq\int_{Q_{\ell}(x_0)}\left[(1-\delta)|(\nabla-i\kappa H(\sigma_{\ell}\overline{B}_{Q_{\ell}(x_0)}\Ab_{0}(x-x_{0}))u|^2-\kappa^2|u|^2+\frac{\kappa^2}{2}|u|^4\right]\,dx\nonumber\\
&\qquad\qquad\qquad\qquad\quad\qquad\qquad\qquad\qquad-C\kappa^{2}\left(\delta^{-\frac1{2}}\ell^{2}\kappa+\delta^{-\frac1{2}}\ell^{\alpha}\right)^{2}\int_{Q_{\ell}(x_0)}|\psi|^2dx\nonumber\\
&\geq (1-\delta)\mathcal
E_0(u,\sigma_{\ell}\overline{B}_{Q_{\ell}(x_0)}\Ab_{0};Q_{\ell}(x_0))-\widehat{C}\kappa^{2}\left(\delta^{-1}\ell^{4}\kappa^{2}+\delta^{-1}\ell^{2\alpha}+\delta\right)\int_{Q_{\ell}(x_0)}|\psi|^2\,dx\,.
\end{align}
This finishes the proof of the lemma.
\qed
\begin{prop}\label{prop-lb}
For all $\alpha\in(0,1)$, there exist positive
constants $C$, $\epsilon_{0}$ and $\kappa_0$ such that, if $(\kappa, H)$ satisfies \eqref{kappaH} holds, $\ell \in (0,\frac 12)$, $\epsilon\in(0,\epsilon_0)$, $(\psi,\Ab)\in H^1(\Omega;\C)\times H^1_{\Div}(\Omega)$ a
critical point of \eqref{eq-2D-GLf}, and
$\overline{Q_{\ell}(x_0)}\subset\left(\Omega \cap \{ |B_0| >\epsilon\}\right)$, then
$$\frac1{|Q_{\ell}(x_0)|}\mathcal E_0(\psi,\Ab;Q_{\ell}(x_0))\geq g\left(\displaystyle\frac{H}{\kappa}\overline{B}_{Q_{\ell}(x_0)}\right)\kappa^2-C\left(\ell^{3}\kappa^{2}+\ell^{2\alpha-1}+(\ell\kappa\epsilon)^{-1}+\ell\epsilon^{-1}\right)\kappa^{2}\,.$$
Here $g(\cdot)$ is the function introduced in \eqref{eq-g(b)},  and $\overline{B}_{Q_{\ell}(x_0)}$ is introduced in \eqref{sup_B0}.
\end{prop}
\begin{proof}~\\
We use Lemma~\ref{lem-lb} and the inequality $\|\psi\|_\infty\leq 1$ to obtain,
\begin{align}
\mathcal E_0(\psi,\Ab;Q_{\ell}(x_0))\geq (1-\delta)\mathcal
E_0(u,\sigma_{\ell}\overline{B}_{Q_{\ell}(x_0)}&\Ab_{0}(x-x_{0});Q_{\ell}(x_0))\nonumber\\
&-C\kappa^{2}\left(\delta^{-1}\ell^{4}\kappa^{2}+\delta^{-1}\ell^{2\alpha}+\delta\right)|Q_{\ell}(x_0)|\,.\label{eq-2D-GLlb1}
\end{align}
Let
\begin{equation}\label{defbR}
b=\displaystyle\frac{H}{\kappa}\overline{B}_{Q_{\ell}(x_0)}\,,\qquad R=\ell\sqrt{\kappa H \overline{B}_{Q_{\ell}(x_0)}}\,.
\end{equation}
Define the rescaled function,
\begin{equation}\label{defv}
v(x)=u\left(\frac{\ell}{R}x+x_{0}\right)\,,\qquad\quad\forall~x\in Q_{R}\,.
\end{equation}
Remember that $\sigma_{\ell}$ denotes the sign of $B_{0}$ in $Q_{\ell}(x_0)$. The change of variable $ y=\frac{R}{\ell}
(x-x_{0})$ gives:
\begin{align}
\mathcal E_{0}(u,\sigma_{\ell}\overline{B}_{Q_{\ell}(x_0)}\Ab_0(x-x_0)&;Q_{\ell}(x_0))\nonumber\\
&=\int_{Q_{R}}\left(\Big|\left(\frac{R}{\ell}\nabla_{y}-i\sigma_{\ell}\frac{R}{\ell}\Ab_{0}(y)\right)v\Big|^{2}-\kappa^{2}|v|^{2}+\frac{\kappa^2}{2}|v|^{4}\right)\,\frac{\ell}{R}dy\nonumber\\
&=\int_{Q_{R}}\left(|(\nabla_{y}-i\sigma_{\ell}\Ab_{0})v|^{2}-\frac{\kappa}{H\overline{B}_{Q_{\ell}(x_0)}}|v|^{2}+\frac{\kappa}{2H\overline{B}_{Q_{\ell}(x_0)}}|v|^{4}\right)\,dy\nonumber\\
&=\frac{\kappa}{H\overline{B}_{Q_{\ell}(x_0)}}\int_{Q_{R}}b\left(|(\nabla_{y}-i\sigma_{\ell}\Ab_0) v|^2-|v|^2+\frac1{2}|v|^4\right)\,dy\nonumber\\
&=\frac1{b}\,G^{\sigma_{\ell}}_{\,b\,,Q_{R}}(v)\,.\label{eq-2D-GLlb2}
\end{align}
We still need to estimate from below the reduced energy
$G^{\sigma_{\ell}}_{\,b\,,Q_{R}}(v)$. Since $v$ is not in $H^1_0(Q_{R})$, we
introduce a cut-off function $\chi_{R}\in C_{c}^{\infty}(\R^{2})$
such that
\begin{equation}\label{defxR}
0\leq\chi_{R}\leq 1\quad{\rm in~}\R^2\,,\quad {\rm supp}\,\chi_{R}\subset Q_{R}\,,\quad  \chi_{R}=1\quad{\rm in~}
Q_{R-1}\,,\quad{\rm and}\quad |\nabla\chi_{R}|\leq M~{\rm
in}~\R^{2}\,.
\end{equation}
The constant $M$ is universal.\\
 Let
\begin{equation}\label{defuR}
u_{R}=\chi_{R}\,v\,.
\end{equation}
We have,
\begin{align}
G^{\sigma_{\ell}}_{b,\,Q_{R}}(v)&=\int_{Q_{R}}\left(b|(\nabla-i\sigma_{\ell}\Ab_0)v|^2-|v|^2+\frac1{2}|v|^4\right)\,dx\nonumber\\
&\geq\int_{Q_{R}}\left(b|\chi_{R}(\nabla-i\sigma_{\ell}\Ab_0)v|^2-|\chi_{R}v|^2+\frac1{2}|v|^4+(\chi_{R}^2-1)|v|^2\right)\,dx\nonumber\\
&\geq G^{\sigma_{\ell}}_{b\,,Q_{R}}(\chi_{R} v)-\int_{Q_{R}}(1-\chi_{R}^2)|v|^2dx-2\int_{Q_{R}}\Big|\langle(\nabla-i\sigma_{\ell}\Ab_0)\chi_{R}
v\,,\,\nabla\chi_{R}v\rangle\Big|\,dy\,.\label{est-of-G}
\end{align}
Having in mind  \eqref{defv} and \eqref{defu}, we get,
$$
\Big|\Big(\nabla_{y}-i\sigma_{\ell}\Ab_0(y)\Big)v(y)\Big|= \frac{\ell}{R}\Big|\Big(\nabla_{x}-i\kappa
H\sigma_{\ell}\overline{B}_{Q_{\ell}(x_0)}\Ab_0(x-x_0)\Big)u(x)\Big|\,.
$$
Using the estimate in \eqref{eq-est1}, \eqref{eq-2D-est-g} and
\eqref{lem-F} we get,
\begin{align}\label{app A_0}
\Big|\Big(\nabla_{y}-i\sigma_{\ell}\Ab_0(y)\Big)v(y)\Big|&\leq\frac{\ell}{R}\Big|\Big(\nabla_{x}-i\kappa
H\sigma_{\ell}\overline{B}_{Q_{\ell}(x_0)}(\Ab+\nabla\varphi)\Big)u(x)\Big|\nonumber\\
&\qquad\qquad\qquad\qquad+\frac{\kappa
H\ell}{R}\Big|(\Ab-\sigma_{\ell}\overline{B}_{Q_{\ell}(x_0)}\Ab_{0}(x-x_0)-\nabla\varphi)u(x)\Big|\nonumber\\
&\leq \frac {C_{1}\ell} {R}\left(\kappa+\kappa\ell^{\alpha}+\kappa^{2}\ell^{2}\right)\,.
\end{align}
From the definition of $u_{R}$ in \eqref{defuR} and $\chi_{R}$ in \eqref{defxR} we get,
\begin{equation}\label{v<1}
|v|\leq1\,.
\end{equation} 
Using \eqref{v<1}, \eqref{app A_0} and the definition of $\chi_{R}$ in \eqref{defxR}, we get:
\begin{equation}\label{A(y)}
\int_{Q_{R}}\Big|\langle(\nabla-i\sigma_{\ell}\Ab_0)\chi_{R}
v\,,\,\nabla\chi_{R}v\rangle\Big|\,dy\leq C_{1}\left(\kappa\ell+\kappa\ell^{\alpha+1}+\kappa^{2}\ell^{3}\right)\,,
\end{equation}
and 
\begin{align}
\int_{Q_{R}}(1-\chi_{R}^2)|v|^2dx&\leq |Q_{R}\setminus Q_{R-1}|\nonumber\\
&\leq R\,.\label{xR-1}
\end{align}
Inserting \eqref{A(y)} and \eqref{xR-1} into \eqref{est-of-G}, we get,
\begin{align*}
G^{\sigma_{\ell}}_{b\,,Q_{R}}(v)&\geq G^{\sigma_{\ell}}_{b\,,Q_{R}}(u_{R})-C_{2}\left(\kappa\ell+\kappa\ell^{\alpha+1}+\kappa^{2}\ell^{3}+\kappa\ell\sqrt{\epsilon}\right)\,,\\
&\geq G^{\sigma_{\ell}}_{b\,,Q_{R}}(u_{R})-C_{2}\left(\kappa\ell(\sqrt{\epsilon}+1)+\kappa^{2}\ell^{3}\right)\,.
\end{align*}
There are two cases:\\

{\it Case 1:}\hskip0.5cm $\sigma_{\ell}=+1$, when $B_{0}>0,\qquad \text{in~} Q_{\ell}(x_0)$.\\

{\it Case 2:}\hskip0.5cm $\sigma_{\ell}=-1$, when $B_{0}<0,\qquad\text{in}~ Q_{\ell}(x_0)$.\\

In Case~1, after recalling the definition of $m_{0}(b,R)$ introduced in \eqref{m_0}, where $b$ introduced in \eqref{defbR} we get,
\begin{equation}\label{eq-2D-GLlb3}
G^{+1}_{b\,,Q_{R}}(v)\geq m_0(b,R)-C_{2}\left(\kappa\ell(\sqrt{\epsilon}+1)+\kappa^{2}\ell^{3}\right)\,.
\end{equation}
We get by collecting the estimates in
\eqref{eq-2D-GLlb1}-\eqref{eq-2D-GLlb3}:
\begin{align}
\frac1{|Q_{\ell}(x_0)|}\mathcal
E_0(\psi,\Ab;Q_{\ell}(x_0))&\geq\frac{(1-\delta)}{b\ell^{2}}\left(m_0(b,R)-C_{2}\left(\kappa\ell+\kappa^{2}\ell^{3}(\epsilon+1)\right)\right)\nonumber\\
&\qquad\qquad\qquad\qquad\qquad\qquad-C\left(\delta^{-1}\ell^{4}\kappa^{2}+\delta^{-1}\ell^{2\alpha}+\delta\right)\kappa^{2}\nonumber\\
&\geq \frac{(1-\delta)}{b\ell^2}
m_0(b,R)-r(\kappa)\label{eq-2D-GLlb4}\,,
\end{align}
 where
\begin{equation}\label{eq-2D-r1}
r(\kappa)=C_{3}\left(\delta^{-1}\ell^{4}\kappa^{4}+\delta^{-1}\ell^{2\alpha}\kappa^{2}+\delta\kappa^{2}
+\frac {1} {b\ell^{2}}\left(\kappa\ell(\sqrt{\epsilon}+1)+\kappa^{2}\ell^{3}\right)\right)\,.
\end{equation}
Theorem \ref{thm-thmd-AS} tells us that $m_0(b,R)\geq R^2g(b)$ for
all $b\in[0,1]$ and $R$ sufficiently large. Here $g(b)$ is
introduced in \eqref{eq-g(b)}. Therefore, we get from
\eqref{eq-2D-GLlb4} the estimate,
\begin{equation}\label{eq-2D-GLlb-b<1}
\frac1{|Q_{\ell}(x_0)|}\mathcal E_0(\psi,\Ab;Q_{\ell}(x_0))\geq
\left(\frac{(1-\delta)R^{2}}{b\ell^{2}}\right) g(b)- r(\kappa)\,,
\end{equation}
with $b$ defined in \eqref{defbR}.
By choosing $\delta=\ell$ and using that $\overline{Q_{\ell}}\subset\{|B_{0}|>\epsilon\}$, we get,
\begin{equation}
r(\kappa)=\mathcal
O\left(\ell^{3}\kappa^{2}+\ell^{2\alpha-1}+\frac1{\epsilon}\Big((\ell\kappa)^{-1}+\ell\Big)\right)\kappa^{2}\,.
\end{equation}
This implies that,
$$\frac1{|Q_{\ell}(x_0)|}\mathcal E_0(\psi,\Ab;Q_{\ell}(x_0))\geq g\left(\displaystyle\frac{H}{\kappa}\overline{B}_{Q_{\ell}(x_0)}\right)\kappa^2-C\left(\ell^{3}\kappa^{2}+\ell^{2\alpha-1}+(\ell\kappa\epsilon)^{-1}+\ell\epsilon^{-1}\right)\kappa^{2}\,.$$
Similarly, in Case~2, according to Remark~\ref{m_0bar}, we get that,
$$G^{-1}_{b\,,Q_{R}}(v)\geq m_0(b,R)-C_{2}\left(\kappa\ell+\kappa^{2}\ell^{3}(\epsilon+1)\right)\,,$$
and the rest of the proof is as for Case~1.
\end{proof}
\section{Proof of Theorem \ref{thm-2D-main}}
\subsection {Upper bound}
\begin{prop}\label{prop-ub-c0}
There exist positive
constants $C$ and $\kappa_0$ such that, if \eqref{kappaH} holds, then the ground state energy $\E0(\kappa,H)$ in \eqref{eq-2D-gs}
satisfies,
$$\E0(\kappa, H)\leq \kappa^2\int_{\Omega}g\left(\frac{H}{\kappa}|B_{0}(x)|\right)\,dx+C\kappa^{\frac{15}{8}}\,.$$
\end{prop}
\begin{proof}
Let $\ell=\ell(\kappa)$ and $\epsilon=\epsilon(\kappa)$ be  positive parameters such that
$\kappa^{-1}\ll \ell\ll1$ and $\kappa^{-1}\ll \epsilon\ll1$ as $\kappa\to\infty$. For some $ \beta \in (0,1)$, $\mu \in(0,1)$ to be determined later,  we will choose
\begin{equation}\label{defellepsilon}
\ell=\kappa^{-\beta}\,,\,\,\,\epsilon=\kappa^{-\mu}\,.
\end{equation}
Consider the  lattice $\Gamma_{\ell} :=\ell\Z\times\ell\Z$ and write for $\gamma \in \Gamma_{\ell}$,
$Q_{\gamma,\ell}=Q_{\ell}(\gamma)$. For any $\gamma\in \Gamma_{\ell}$ such that  $\overline{Q_{\gamma,\ell}}\subset\Omega  \cap \{|B_{0}|>\epsilon\}$
let
\begin{equation}\label{sousB}
\underline{B}_{\gamma,\ell}=\displaystyle\inf_{x\in
Q_{\gamma,\ell}}|B_{0}(x)|\,.
\end{equation}
Let
$$ \mathcal
I_{\ell,\epsilon}=\Big\{\gamma~:~\overline{Q_{\gamma,\ell}}\subset\Omega  \cap \{|B_{0}|>\epsilon\}\Big\}\,,$$

$$N={\rm Card}\,\mathcal I_{\ell,\epsilon}\,,$$
and
$$\Omega_{\ell,\epsilon}=\text{int}\left(\displaystyle{\cup_{\gamma\in \mathcal
I_{\ell,\epsilon}}}\overline{Q_{\gamma,\ell}}\right)\,.
$$
It follows from \eqref{B(x)} that:
$$ N=|\Omega|\ell^{-2}+ \mathcal O (\epsilon \ell^{-2}) +{\mathcal O}(\ell^{-1}) \mbox{ as } \ell\to 0 \mbox{ and } \epsilon \to 0\,.
$$
Let
\begin{equation}\label{def-bR}
b=\displaystyle\frac{H}{\kappa}\underline{B}_{\gamma,\ell}\,,\, R=\ell\sqrt{\kappa H \underline{B}_{\gamma,\ell}}\,,
\end{equation}
and $u_{R}$ be a minimizer of the functional in \eqref{eq-LF-2D}, i.e.
$$m_{0}(b,R)=\int_{Q_{R}}\left(b|(\nabla-i\Ab_{0})u_{R}|^2-|u_{R}|^2+\frac12|u_{R}|^4\right)\,dx\,.$$
We will need the function $\varphi_{\gamma}$ introduced in Lemma~\ref{app F} which satisfies
$$|\Fb(x)-\sigma_{\gamma,\ell}\underline{B}_{\gamma,\ell}\Ab_{0}(x-\gamma)-\nabla\varphi_{\gamma}(x)|\leq
C\ell^{2}\,,\,\,~\rm{in}\,\,Q_{\gamma,\ell}\,,$$ where $\sigma_{\gamma,\ell}$ is the sign of $B_{0}$ in $Q_{\gamma,\ell}$.\\
We define the function,
$$v(x)=\begin{cases}e^{-i\kappa H \varphi_{\gamma}}u_{R}\Big(\displaystyle\frac{R}{\ell}(x-\gamma)\Big) &\text{if}\,\, x \in  Q_{\gamma,\ell}\subset\{B_{0}>\epsilon\}\\
e^{-i\kappa H \varphi_{\gamma}}\overline{u_{R}}\Big(\displaystyle\frac{R}{\ell}(x-\gamma)\Big) &\text{if}\,\, x \in  Q_{\gamma,\ell}\subset\{B_{0}<-\epsilon\}\\
0&\text{if}\,\, x \in  \Omega\setminus\Omega_{\ell,\epsilon}
\end{cases}
\,.$$
Since $u_{R}\in H^1_0(Q_{R})$, then $v \in H^1(\Omega)$. We compute the energy
of the configuration $(v,\Fb)$. We get,
\begin{align}\label{eq-ub-C0}
\mathcal E(v,\Fb)&=\int_{\Omega}\left(|(\nabla-i\kappa
H\Fb)v|^2-\kappa^2|v|^2+\frac{\kappa^2}2|v|^4\right)\,dx\nonumber\\
&=\sum_{\gamma \in {\mathcal{I_{\ell,\epsilon}}}}\mathcal
E_{0}(v,\Fb;Q_{\gamma,\ell})\,.
\end{align}
We estimate the term $\mathcal E_0(v,\Fb;Q_{\gamma,\ell})$ from
above and we write:
\begin{align}
\mathcal
E_0(v,\Fb;Q_{\gamma,\ell})&=\int_{Q_{\gamma,\ell}}|(\nabla-i\kappa H\Fb)v|^2-\kappa^2|v|^2+\frac{\kappa^2}{2}|v|^4\,dx\nonumber\\
&=\int_{Q_{\gamma,\ell}}\Big|\Big(\nabla-i\kappa
H\big(\sigma_{\gamma,\ell}\underline{B}_{\gamma,\ell}\Ab_{0}(x-\gamma)+\nabla\varphi_{\gamma}(x)\big)\Big)v\nonumber\\
&\qquad\qquad\qquad-i\kappa H\big(\Fb-\sigma_{\gamma,\ell}\underline{B}_{\gamma,\ell}\Ab_{0}(x-\gamma)-\nabla\varphi_{\gamma}(x)\big)\Big)v\Big|^2-\kappa^2|v|^2+\frac{\kappa^2}{2}|v|^4\,dx\nonumber\\
&\leq\int_{Q_{\gamma,\ell}}(1+\delta)\Big|\Big(\nabla-i\kappa
H\big(\sigma_{\gamma,\ell}\underline{B}_{\gamma,\ell}\Ab_{0}(x-\gamma)+\nabla\varphi_{\gamma}(x)\big)\Big)v\Big|^2-\kappa^2|v|^2+\frac{\kappa^2}{2}|v|^4dx\nonumber\\
&\qquad\qquad\quad+C(1+\delta^{-1})(\kappa H)^2\int_{Q_{\gamma,\ell}}\Big|\big(\Fb-\sigma_{\gamma,\ell}\underline{B}_{\gamma,\ell}\Ab_{0}(x-\gamma)-\nabla\varphi_{\gamma}(x)\big)\Big)v\Big|^2\,dx\nonumber\\
&\leq(1+\delta)\mathcal E_0(e^{-i\kappa
H\varphi_{\gamma}}v,\sigma_{\gamma,\ell}\underline{B}_{\gamma,\ell}\Ab_{0}(x-\gamma);Q_{\gamma,\ell})+C(\delta\kappa^{2}+\delta^{-1}\kappa^{4}\ell^{4})\int_{Q_{\gamma,\ell}}|v|^{2}d\,x\label{E2-ub}\,.
\end{align}
Having in mind that $u_{R}$ is a minimizer of the functional in (\ref{eq-LF-2D}), and using the estimate in (\ref{eq-psi<1}) we get:
$$\int_{Q_{\gamma,\ell}}|v|^{2}d\,x\leq |Q_{\gamma,\ell}|\,.$$ 
Remark~\ref{m_0bar} and a change of variables give us,
$$\int_{Q_{\gamma,\ell}}\left(|(\nabla-i\kappa H\sigma_{\gamma,\ell}(\underline{B}_{\gamma,\ell}\Ab_{0}(x-\gamma))e^{-i\kappa
H\varphi_{\gamma}}v|^2-\kappa^2|v|^2+\frac{\kappa^2}2|v|^4\right)\,dx
=\frac{m_0(b,R)}{b}\,.$$
We insert this into \eqref{E2-ub} to obtain,
\begin{equation}\label{E0<m0}
\mathcal E_0(v,\Fb;Q_{\gamma,\ell})\leq (1+\delta)\frac{m_0(b,R)}{b}+C(\delta\kappa^{2}+\delta^{-1}\kappa^{4}\ell^{4})\ell^{2}\,.
\end{equation}
We know from Theorem~\ref{thm-thmd-AS} that $m_0(b,R)\leq g(b)
R^2+CR$ for all $b\in[0,1]$ and $R$ sufficiently large, where $b$ introduced in \eqref{def-bR}. We choose $\delta=\ell$ in \eqref{E0<m0}. That way we get,
\begin{equation}\label{E0<g}
\mathcal E_0(v,\Fb;Q_{\gamma,\ell})\leq g\left(\frac{H}{\kappa}\underline{B}_{\gamma,\ell}\right)\ell^{2}\kappa^{2}+C\left(\frac{1}{\kappa\ell\sqrt{\epsilon}}+\ell+\kappa^{2}\ell^{3}\right)\ell^{2}\kappa^2\,.
\end{equation}
Summing \eqref{E0<g} over $\gamma$ in $I_{\ell,\epsilon}\,,$ we recognize the lower Riemann sum of $x\rightarrow g\left(\displaystyle\frac{H}{\kappa}|B_{0}(x)|\right)$. By monotonicity of $g$, $g$ is Riemann-integrable and its integral is larger than any lower Riemann sum. Thus:
\begin{equation}
\mathcal E(v,\Fb)\leq\left(\int_{\Omega_{\ell,\epsilon}}g\left(\frac{H}{\kappa}|B_{0}(x)|\right)\,dx\right)\kappa^{2}+C\left(\frac{1}{\kappa\ell\sqrt{\epsilon}}+\ell+\kappa^{2}\ell^{3}\right)\kappa^2\,.
\end{equation}
Notice that  using the regularity of  $\partial \Omega$ and \eqref{B(x)}, there exists $C>0$ such that:
\begin{equation}\label{estaire}
|\Omega\setminus\Omega_{\ell,\epsilon}|=\mathcal{O}\left(\ell|\partial\Omega|+C\epsilon\right)\,,
\end{equation}
as $\epsilon$ and $\ell$ tend to $0$.\\

Thus, we get by using the properties of $g$ in Theorem~\ref{thm-thmd-AS},
$$
\int_{\Omega_{\ell,\epsilon}}g\left(\frac{H}{\kappa}|B_{0}(x)|\right)\,dx\leq\int_{\Omega}g\left(\frac{H}{\kappa}|B_{0}(x)|\right)\,dx+\frac{1}{2}|\Omega\setminus\Omega_{\ell,\epsilon}|.
$$
This implies that,
\begin{align}\label{est-en-v,F}
\mathcal E(v,\Fb)&\leq\int_{\Omega}g\left(\frac{H}{\kappa}|B_{0}(x)|\right)\,dx+C\left(\frac{1}{\kappa\ell\sqrt{\epsilon}}+\ell+\epsilon+\kappa^{2}\ell^{3}\right)\kappa^2\,.
\end{align}
We choose in \eqref{defellepsilon}
\begin{equation}\label{beta-mu}
\beta=\frac{3}{4}~{\rm and}~ \mu=\frac{1}{8}.
\end{equation}
With this choice, we infer from \eqref{est-en-v,F},
\begin{equation}
\mathcal E(v,\Fb)\leq\int_{\Omega}g\left(\frac{H}{\kappa}|B_{0}(x)|\right)\,dx+C_{1}\kappa^{\frac{15}{8}}\,.
\end{equation}
This finishes the proof of Proposition~\ref{prop-ub-c0}.
\end{proof}

\begin{rem}
In the case when $B_{0}$ does not vanish in $\Omega$, $\epsilon$ disappears and $\{x\in\Omega; |B_{0}(x)|>0\}=\Omega$. Consequently, the Ginzburg-Lundau energy of $(v,\Fb)$ in \eqref{eq-GLen-D} satisfies:

$$\mathcal E(v,\Fb)\leq\int_{\Omega}g\left(\frac{H}{\kappa}|B_{0}(x)|\right)\,dx+C\left(\frac{1}{\kappa\ell}+\ell+\kappa^{2}\ell^{3}\right)\kappa^2\,.$$

We take the same choice of $\beta$ as in \eqref{beta-mu}, then the ground state energy $\E0(\kappa,H)$ in \eqref{eq-2D-gs}
satisfies,
$$\E0(\kappa, H)\leq \kappa^2\int_{\Omega}g\left(\frac{H}{\kappa}|B_{0}(x)|\right)\,dx+C\kappa^{\frac{7}{4}}\,.$$
\end{rem}

\subsection{Lower bound}\label{lower bound}
We now establish a lower bound for the ground state energy $\E0(\kappa,H)$ in \eqref{eq-2D-gs}.
The parameters $\epsilon$ and $\ell$ have the same form as in \eqref{beta-mu}.\\
Let
\begin{equation}\label{overB(x)}
\overline{B}_{\gamma,\ell}=\sup_{x\in Q_{\gamma,\ell}}|B_{0}(x)|\,,
\end{equation}
and
\begin{equation}\label{def-b-R}
b_{\gamma,\ell}=\displaystyle\frac{H}{\kappa}\overline{B}_{\gamma,\ell}\,,\, R=\ell\sqrt{\kappa H \overline{B}_{\gamma,\ell}}\,,
\end{equation}
If
$(\psi,\Ab)$ is a minimizer of \eqref{eq-2D-GLf}, we have,
\begin{equation}\label{eq-glob-en}
\E0(\kappa,H)=\mathcal E_0(\psi,\Ab;\Omega_{\ell,\epsilon})+\mathcal
E_0(\psi,\Ab;\Omega\setminus\Omega_{\ell,\epsilon}) + (\kappa H)^2
\int_{\Omega} |\curl\big(\Ab - \Fb\big)|^2\,dx \,,
\end{equation}
where, for any $D\subset\Omega$, the energy $\mathcal
E_0(\psi,\Ab;D)$ is introduced in \eqref{eq-GLe0}. Since the magnetic energy term is positive, we may write,
\begin{equation}\label{eq-glob-en1}
\E0(\kappa,H)\geq\mathcal E_0(\psi,\Ab;\Omega_{\ell,\epsilon})+\mathcal
E_0(\psi,\Ab;\Omega\setminus\Omega_{\ell,\epsilon})\,.
\end{equation}
Thus, we get by using \eqref{eq-psi<1},
\eqref{eq-est1}, and \eqref{estaire}:
\begin{align}
 \left|\mathcal E_0(\psi,\Ab;\Omega\setminus\Omega_{\ell,\epsilon})\right|&\leq\int_{\Omega\setminus\Omega_{\ell,\epsilon}}|(\nabla-i\kappa
H\Ab)\psi|^2+\kappa^2|\psi|^2+\frac{\kappa^2}{2}|\psi|^4\,dx\nonumber\\
&\leq|\Omega\setminus\Omega_{\ell,\epsilon}|\left(C_1\kappa^{2}\|\psi\|_{L^{\infty}(\Omega)}^{2}+\kappa^{2}\|\psi\|_{L^{\infty}(\Omega)}^{2}+\frac{\kappa^{2}}{2}\|\psi\|_{L^{\infty}(\Omega)}^{4}\right)\nonumber\\
&\leq C_2(\ell+\epsilon)\kappa^2\,\label{eq-en-bnd}.
\end{align}
To estimate $\mathcal E_0(\psi,\Ab;\Omega_{\ell,\epsilon})$, we notice that,
$$\mathcal E_0(\psi,\Ab;\Omega_{\ell,\epsilon})=\sum_{\gamma\in\mathcal I_{\ell,\epsilon}}\mathcal E_0(\psi,\Ab;Q_{\gamma,\ell})\,.$$
Using Proposition~\ref{prop-lb} with $\alpha=\frac{2}{3}$ and \eqref{eq-en-bnd} with $\beta=\frac{3}{4}$ and $\mu=\frac{1}{8}$ in \eqref{defellepsilon}, we get,
\begin{align}
\mathcal E_0(\psi,\Ab;\Omega_{\ell,\epsilon})&\geq\sum_{\gamma\in\mathcal
I_{\ell,\epsilon}}
g\left(\frac{H}{\kappa}\overline{B}_{Q_{\ell}(x_0)}\right)\ell^2\kappa^2-C\left(\ell^{3}\kappa^{2}+\ell^{2\alpha-1}+(\ell\kappa\epsilon)^{-1}+\ell\epsilon^{-1}\right)\kappa^2\nonumber\\
&\geq\kappa^{2}\sum_{\gamma\in\mathcal
I_{\ell,\epsilon}}g\left(\frac{H}{\kappa}\overline{B}_{Q_{\ell}(x_0)}\right)\ell^2-C_{1}\kappa^{\frac{15}{8}}\nonumber,
\end{align}
and
\begin{equation}\label{final-est-error}
\mathcal E_0(\psi,\Ab;\Omega\setminus\Omega_{\ell,\epsilon})\geq-C_{2}\kappa^{\frac{15}{8}}\,.
\end{equation}
As for the upper bound, we can use the monotonicity of $g$ and recognize that the sum above is an  upper Riemann sum of $g$. In that way, we get,
$$
\mathcal
E_0(\psi,\Ab;\Omega)\geq\kappa^{2}\int_{\Omega_{\ell,\epsilon}}g\left(\frac{H}{\kappa}|B_{0}(x)|\right)\,dx-C_{1}\kappa^{\frac{15}{8}}\,.
$$
Recalling the assumption $\Omega_{\ell,\epsilon}\subset\Omega$ and that $g\leq0$, we deduce that,
\begin{equation}\label{eq-en-blk}
\mathcal E_0(\psi,\Ab;\Omega)\geq\kappa^{2}\int_{\Omega}g\left(\frac{H}{\kappa}|B_{0}(x)|\right)\,dx-C_{1}\kappa^{\frac{15}{8}}\,.
\end{equation}
Putting \eqref{final-est-error} and \eqref{eq-en-blk} into \eqref{eq-glob-en1}, which finishes the proof of Theorem~\ref{thm-2D-main}.
\begin{rem}\label{rm-lower}
When $B_{0}$ does not vanish, the local energy in $Q_{\ell}(x_0)$ in Proposition~\ref{prop-lb} becomes:
$$\frac1{|Q_{\ell}(x_0)|}\mathcal E_0(\psi,\Ab;Q_{\ell}(x_0))\geq g\left(\displaystyle\frac{H}{\kappa}\overline{B}_{Q_{\ell}(x_0)}\right)\kappa^2-C\left(\ell^{3}\kappa^{2}+\ell^{2\alpha-1}+(\ell\kappa)^{-1}\right)\kappa^{2}\,.$$
Similarly, using $\alpha=\frac{2}{3}$, $\ell=\kappa^{-\frac{3}{4}}$ and replacing the upper Riemann sum  by the integral, we get:
$$\E0(\kappa,H)\geq\kappa^{2}\int_{\Omega}g\left(\frac{H}{\kappa}|B_{0}(x)|\right)\,dx-C\kappa^{\frac{7}{4}}\,.$$
\end{rem}
\subsection{Proof of Corollary \ref{corol-2D-main}}
If $(\psi,\Ab)$ is a minimizer of \eqref{eq-2D-GLf}, we have,
\begin{equation}
\mathcal E(\psi,\Ab;\Omega)=\mathcal E_{0}(\psi,\Ab;\Omega)+(\kappa H)^{2}\int_{\Omega}|\curl(\Ab-\Fb)|^{2}\,dx.
\end{equation}
Remark~\ref{rm-lower} and \eqref{eq-en-blk} tell us that,
\begin{equation}\label{lower-E_0}
\mathcal
E_0(\psi,\Ab;\Omega)\geq\kappa^{2}\int_{\Omega_{\ell,\epsilon}}g\left(\frac{H}{\kappa}|B_{0}(x)|\right)\,dx-C\kappa^{\tau_0}\,.
\end{equation}
Using Theorem~\ref{thm-2D-main} and \eqref{lower-E_0} we get,
$$
\kappa^{2}\int_{\Omega}g\left(\frac{H}{\kappa}|B_{0}(x)|\right)\,dx-C\kappa^{\tau_0}+(\kappa H)^{2}\int_{\Omega}|\curl(\Ab-\Fb)|^{2}\,dx\leq\kappa^{2}\int_{\Omega}g\left(\frac{H}{\kappa}|B_{0}(x)|\right)\,dx+C_{2}\kappa^{\tau_{0}}.
$$
This implies that,
\begin{equation}\label{eq-mag-en}
(\kappa H)^{2}\int_{\Omega}|\curl(\Ab-\Fb)|^{2}\,dx\leq C'\kappa^{\tau_{0}}\,.
\end{equation}
\section{Local Energy Estimates}
The object of this section is to give an estimates to the Ginzburg-Landau energy \eqref{eq-GLen-D} in the open set $D\subset\Omega$.
\subsection{Main statements}
\begin{thm}\label{thm-2D-mainloc}
There exist positive
constants $\kappa_0$ such that if \eqref{kappaH} is true and $D \subset \Omega$ is an open set, then the local energy  of the minimizer  satisfies,
\begin{equation}\label{eq-2Dl-thm}
 \left|\mathcal E(\psi,\Ab;\mathcal D) -\kappa^2\int_{D}g\left(\displaystyle
\frac{H}{\kappa}\,|B_{0}(x)|\right)\,dx\right| = o(\kappa^2)\,.
\end{equation}
\end{thm}

For all $(\ell,x_0)$ such that $\overline{Q_{\ell}(x_0)}\subset\Omega\cap\{|B_{0}|>\epsilon\}$, we define
\begin{equation}\label{infB}
\underline{B}_{Q_{\ell}(x_0)}=\displaystyle\inf_{x\in Q_{\ell}(x_0)}|B_{0}(x)|\,,
\end{equation}
where $B_{0}$ is introduced in \eqref{B(x)}.
\begin{prop}\label{prop-ub}
For all $\alpha\in(0,1)$, there exist positive
constants $C$ and $\kappa_0$ such that if \eqref{kappaH} is true, $\ell \in (0,\frac 12)$, $(\psi,\Ab)\in H^1(\Omega;\C)\times H^1_{\Div}(\Omega)$ is
a minimizer of \eqref{eq-2D-GLf}, and $\overline{Q_{\ell}(x_0)}\subset\Omega\cap\{|B_0|>\epsilon\} $, then,
$$\frac1{|Q_{\ell}(x_0)|}\mathcal E_0(\psi,\Ab;Q_{\ell}(x_0))\leq g\left(\displaystyle\frac{H}{\kappa}\underline{B}_{Q_{\ell}(x_0)}\right)\kappa^2+C\left(\ell^{3}\kappa^{2}+\ell^{2\alpha-1}+(\ell\kappa\sqrt{\epsilon})^{-1}\right)\kappa^2\,.$$
Here $g(\cdot)$ is the function introduced in \eqref{eq-g(b)} and $\mathcal E_0$ is
the functional in \eqref{eq-GLe0}.
\end{prop}
\begin{proof}
As explained earlier in the proof of Lemma~\ref{lem-lb} in \eqref{A-F}, we may
suppose after performing a gauge transformation that the magnetic
potential $\Ab$ satisfies,
\begin{equation}\label{eq-2D-est-g'}
|\Ab(x)-\Fb(x)|\leq C\frac{\ell^{\alpha}}{H}\,,\quad\forall~x\in Q_{\ell}(x_0)\,.
\end{equation}
Let
\begin{equation}\label{defb-R}
b=\displaystyle\frac{H}{\kappa}\underline{B}_{Q_{\ell}(x_0)},\,\,\,R=\ell\sqrt{\kappa H\underline{B}_{Q_{\ell}(x_0)}}\,,
\end{equation}
and $u_{R}\in H^1_0(Q_{R})$ be the minimizer of the functional $G^{+1}_{b,Q_{R}}$
introduced in \eqref{eq-LF-2D}. Let $\chi_{R}\in
C_c^\infty(\R^2)$ be a cut-off function such that,
$$0\leq\chi_{R}\leq 1\quad{\rm in~}\R^2\,,\quad {\rm supp}\,\chi_{R}\subset Q_{R+1}\,,\quad  \chi_{R}=1\quad{\rm in~}
Q_{R}\,,$$ and  $|\nabla\chi_{R}|\leq C$
for some universal constant $C$.\\
Let $\eta_{R}(x)=1-\chi_{R}\left(\displaystyle\frac{R}{\ell}(x-x_0)\right)$ for all $x\in\R^2$ and $\widetilde{\ell}=\ell\left(1+\displaystyle\frac{1}{R}\right)$.\\
This implies that,
\begin{align}
&\eta_{R}(x)=0 \qquad\qquad {\rm in}\quad Q_{\ell}(x_0)\label{eta0}\\
&0\leq\eta_{R}(x)\leq1\qquad\, {\rm in}\quad Q_{\widetilde{\ell}}(x_0)\setminus Q_{\ell}(x_0)\label{0eta1}\\
&\eta_{R}(x)=1\qquad\qquad {\rm  in}\quad\Omega\setminus
Q_{\widetilde{\ell}}(x_0)\label{eta1}\,.
\end{align}
Consider the function $w(x)$ defined as follows,
$$w(x)=\eta_{R}(x)\psi(x)\,\,\,\,\,\,\,\, ~{\rm in} ~\,\,\Omega\setminus Q_{\ell}(x_0),$$
and, if $x\in Q_{\ell}(x_{0})$,
$$w(x)=\left\{
  \begin{array}{ll}
    e^{i\kappa H\varphi}u_{R}\left(\displaystyle\frac{R}{\ell}(x-x_0)\right) &\quad {\rm if}\quad Q_{\ell}(x_0)\subset\{B_{0}>\epsilon\} \\
    e^{i\kappa H\varphi}\overline{u}_{R}\left(\displaystyle\frac{R}{\ell}(x-x_0)\right) &\quad {\rm if}\quad Q_{\ell}(x_0)\subset\{B_{0}<-\epsilon\} \,.
  \end{array}
\right.$$
Notice that by construction, $w=\psi$ in $\Omega\setminus
Q_{\widetilde{\ell}}(x_0)$. We will
prove that, for any $\delta\in(0,1)$ and $\alpha\in(0,1)$,
\begin{equation}\label{eq-GLub}
\mathcal E(w,\Ab;\Omega)\leq \mathcal E(\psi,\Ab;\Omega\setminus
Q_{\ell}(x_0))+(1+\delta)\frac{\ell}{bR} m_0(b,R)+ r_0(\kappa)\ell^{2}\,,
\end{equation}
and for some constant
$C$, $r_0(\kappa)$ is given as follows,
\begin{equation}\label{eq-ub-r0}
r_0(\kappa)=C\left(\delta+\delta^{-1}\ell^4\kappa^{2}+\delta^{-1}\ell^{2\alpha}+\frac{1}{\ell\kappa\sqrt{\epsilon}}\right)\kappa^{2}\,.
\end{equation}
{\bf Proof of \eqref{eq-GLub}:}
With $\mathcal E_0$  defined in
\eqref{eq-GLe0}, we write,
\begin{equation}\label{eq-E=E1+E2}
\mathcal E_0(w,\Ab;\Omega)=\mathcal E_1+\mathcal
E_2\,,\end{equation} where
\begin{equation}\label{eq-E1+E2}
\mathcal E_1=\mathcal E_0(w,\Ab;\Omega\setminus Q_{\ell}(x_0))\,,\quad
\mathcal E_2=\mathcal E_0(w,\Ab;Q_{\ell}(x_0))\,.
\end{equation}
We estimate $\mathcal E_1$ and $\mathcal E_2$ from above.
Starting with $\mathcal E_1$ and using \eqref{eta1}, we get,
\begin{align}
\mathcal E_1&=\int_{\Omega\setminus Q_{\ell}(x_0)}|(\nabla-i\kappa
H\Ab)\eta_{R}\psi|^{2}-\kappa^{2}|\eta_{R}\psi|^2+\frac{\kappa^2}{2}|\eta_{R}\psi|^4\,dx\nonumber\\
&=\int_{\Omega\setminus Q_{\ell}(x_0)}\eta_{R}^{2}|(\nabla-i\kappa
H\Ab)\psi|^{2}+|\nabla\eta_{R}\psi|^{2}+2R\langle\eta_{R}(\nabla-i\kappa
H\Ab)\psi\,,\,\nabla\eta_{R}\psi\rangle\nonumber\\
&\quad\quad\quad\quad\quad\quad\quad\quad\quad\quad\quad\quad\quad\quad\quad\quad\quad\quad\quad\quad\quad\quad\quad\quad\quad-\kappa^{2}\eta_{R}^2|\psi|^2+\frac{\kappa^2}{2}\eta_{R}^{4}|\psi|^4\,dx\nonumber\\
&=\mathcal E_0(\psi,\Ab;\Omega\setminus Q_{\ell}(x_0))+\mathcal
R(\psi,\Ab)\label{eq-E1-ub}\,,
\end{align}
where
\begin{multline*}
\mathcal R(\psi,\Ab)=\int_{Q_{\widetilde{\ell}}(x_0)\setminus
Q_{\ell}(x_0)}\bigg{(}(\eta_{R}^2-1)\left(|(\nabla-i\kappa H\Ab)\psi|^2-\kappa^2|\psi|^2\right)+|\psi\nabla\eta_{R}|^2+\frac{\kappa^2}2(\eta_{R}^4-1)|\psi|^4\\
+2\Re\langle  \eta_{R}(\nabla-i\kappa
H\Ab)\psi,\psi\nabla\eta_{R}\rangle\bigg{)}\,dx\,.
\end{multline*}
Noticing that $\left|Q_{\widetilde{\ell}}(x_0)\setminus
Q_{\ell}(x_0)\right|\leq\frac{\ell}{\sqrt{\kappa
H\underline{B}_{Q_{\ell}(x_0)}}}$ and using \eqref{0eta1} together with the
estimates in \eqref{eq-psi<1}, \eqref{kappaH}, \eqref{eq-est1} and $|\nabla\eta_{R}|\leq C\displaystyle\frac{R}{\ell}$, we get,
\begin{equation}\label{eq-est-R-ub}
|\mathcal R(\psi,\Ab)|\leq C\frac{\ell\kappa}{\sqrt{\epsilon}}\,.
\end{equation}
Inserting \eqref{eq-est-R-ub} in \eqref{eq-E1-ub}, we get the
following estimate,
\begin{equation}\label{eq-E1-ub1}
\mathcal E_1\leq\mathcal E_0(\psi,\Ab;\Omega\setminus
Q_{\ell}(x_0))+C\frac{\ell\kappa}{\sqrt{\epsilon}}\,.
\end{equation}

We estimate the term $\mathcal E_2$ in \eqref{eq-E1+E2}. We will
need the function $\varphi_{0}$ introduced in
Lemma~\ref{app F} and satisfying
$|\Fb(x)-\sigma_{\ell}\underline{B}_{Q_{\ell}(x_0)}\Ab_{0}(x-x_{0})-\nabla\varphi_{0}(x)|\leq
C\ell^{2}$ in $Q_{\ell}(x_0)$, where $\sigma_{\ell}$ denotes the sign of $B_{0}$. We start with the
kinetic energy term and write for any $\delta\in(0,1)$:
\begin{align}
\mathcal E_2&=\int_{Q_{\ell}(x_0)}\Big|\Big(\nabla-i\kappa
H(\sigma_{\ell}\underline{B}_{Q_{\ell}(x_0)}\Ab_{0}(x-x_{0})+\nabla\varphi(x))\Big)w\nonumber\\
&\qquad\qquad\qquad-i\kappa H\Big(\Ab-\big(\sigma_{\ell}\underline{B}_{Q_{\ell}(x_0)}\Ab_{0}(x-x_{0})+\nabla\varphi(x)\big)\Big)\Big|^2+\left(-\kappa^2|w|^2+\frac{\kappa^2}{2}|w|^4\right)\,dx\nonumber\\
&\leq\int_{Q_{\ell}(x_0)}(1+\delta)\Big|\Big(\nabla-i\kappa
H(\sigma_{\ell}\underline{B}_{Q_{\ell}(x_0)}\Ab_{0}(x-x_{0})+\nabla\varphi(x))\Big)w\Big|^2-\kappa^2|w|^2+\frac{\kappa^2}{2}|w|^4\,dx\nonumber\\
&\,+(1+\delta^{-1})(\kappa H)^{2}\int_{Q_{\ell}(x_0)}\Big|(\Ab-\nabla\phi_{x_{0}}-\Fb)w+(\Fb-\sigma_{\ell}\underline{B}_{Q_{\ell}(x_0)}\Ab_{0}(x-x_{0})-\nabla\varphi_{0}(x))w\Big|^2\,dx\label{E2-ub}\,.
\end{align}
Using the estimate in \eqref{eq-2D-est-g'} together with \eqref{kappaH} and \eqref{eq-psi<1}, we deduce  the upper bound,
\begin{equation}
\label{E2-ub1} \mathcal E_2\leq (1+\delta)
\mathcal E_0(e^{-i\kappa H\varphi}w,\sigma_{\ell}\underline{B}_{Q_{\ell}(x_0)}\Ab_{0}(x-x_{0}); Q_{\ell}(x_0))
+C(\delta^{-1}\ell^{2\alpha}+\delta^{-1}\ell^{4}\kappa^{2}+\delta)\kappa^{2}\ell^{2}\,,
\end{equation}
where $\alpha\in(0,1)$.\\
There are two cases:\\
\textbf{Case 1:} If $B_{0}>\epsilon$ in $Q_{\ell}(x_0)$, then $\sigma_{\ell}=+1$ and 
$$w(x)=\left\{
  \begin{array}{ll}
    e^{i\kappa H\varphi}u_{R}\left(\displaystyle\frac{R}{\ell}(x-x_0)\right) &\quad {\rm in}\quad Q_{\ell}(x_0) \\
    \eta_{R}(x)\psi(x) & \quad {\rm in}\quad\Omega\setminus
Q_{\ell}(x_0)\,.
  \end{array}
\right.$$
The change of variable $y=\displaystyle\frac{R}{\ell}
(x-x_0)$  and \eqref{defbR} gives us:
\begin{align}
\mathcal E_{0}(e^{-i\kappa H\varphi}w,\sigma_{\ell}\underline{B}_{Q_{\ell}(x_0)}&\Ab_0(x-x_0);Q_{\ell}(x_0))\nonumber\\
&=\int_{Q_{R}}\left(\Big|\left(\frac{R}{\ell}\nabla_{y}-i\frac{R}{\ell}\Ab_{0}(y)\right)u_{R}\Big|^{2}-\kappa^{2}|u_{R}|^{2}+\frac{\kappa^2}{2}|u_{R}|^{4}\right)\,\frac{\ell}{R}dy\nonumber\\
&=\int_{Q_{R}}\left(|(\nabla_y-i\Ab_{0}(y))u_{R}|^{2}-\frac{\kappa}{H\underline{B}_{Q_{\ell}(x_0)}}|u_{R}|^{2}+\frac{\kappa}{2H\underline{B}_{Q_{\ell}(x_0)}}|u_{R}|^{4}\right)\,dy\nonumber\\
&=\frac{\kappa}{H\underline{B}_{Q_{\ell}(x_0)}}\int_{Q_{R}}b\left(|(\nabla_{y}-i\Ab_0(y))u_{R}|^2-|u_{R}|^2+\frac1{2}|u_{R}|^4\right)\,dy\nonumber\\
&=\frac1{b}\,G^{+1}_{\,b\,,Q_{R}}(u_{R})\,,\label{eq-2D-GLlb2}
\end{align}
where $G^{+1}_{b,Q_{R}}$ is the functional from \eqref{eq-LF-2D}.\\
\textbf{Case 2:} If $B_{0}<-\epsilon$ in $Q_{\ell}(x_0)$, then $\sigma_{\ell}=-1$ and 
$$w(x)=\left\{
  \begin{array}{ll}
    e^{i\kappa H\varphi}\overline{u}_{R}\left(\displaystyle\frac{R}{\ell}(x-x_0)\right) &\quad {\rm in}\quad Q_{\ell}(x_0) \\
    \eta_{R}(x)\psi(x) & \quad {\rm in}\quad\Omega\setminus
Q_{\ell}(x_0)\,.
  \end{array}
\right.$$
Similarly, like in case 1, we have,
$$\mathcal E_{0}(e^{-i\kappa H\varphi}w,\sigma_{\ell}\underline{B}_{Q_{\ell}(x_0)}\Ab_{0}(x-x_0);Q_{\ell}(x_0))=\frac1{b}\,G^{-1}_{\,b\,,Q_{R}}(\overline{u}_{R})=\frac1{b}\,G^{+1}_{\,b\,,Q_{R}}(u_{R})\,.$$
In both cases we see that,
\begin{equation}\label{eq-2D-GLlb2}
\mathcal E_{0}(e^{-i\kappa H\varphi}w,\sigma_{\ell}\underline{B}_{Q_{\ell}(x_0)}\Ab_{0}(x-x_0);Q_{\ell}(x_0))=\frac1{b}\,G^{+1}_{\,b\,,Q_{R}}(u_{R})=\frac{m_0(b, R)}{b}.
\end{equation}
Inserting \eqref{eq-2D-GLlb2} into \eqref{E2-ub1}, we get,
\begin{equation}\label{E1-ub3}
\mathcal E_2\leq (1+\delta)\frac1{b}
m_0(b, R)+C(\delta+\delta^{-1}\ell^{4}\kappa^2+\delta^{-1}\ell^{2\alpha})\kappa^{2}\ell^{2}\,.
\end{equation}
Inserting \eqref{eq-E1-ub1} and \eqref{E1-ub3} into
\eqref{eq-E=E1+E2}, we deduce that,
\begin{equation}\label{eq-ub-E1+E2}
\mathcal E_0(\varphi,\Ab)\leq  \mathcal E_0(\psi,\Ab;\Omega\setminus
Q_{\ell}(x_0))+(1+\delta)\displaystyle\frac1{b} m_0(b, R)
+C(\delta+\delta^{-1}\ell^{4}\kappa^2+\delta^{-1}\ell^{2\alpha}\kappa^{2}+(\ell\kappa\sqrt{\epsilon})^{-1})\ell^2\kappa^2\,.
\end{equation}
This proves \eqref{eq-GLub}. Now, we show how \eqref{eq-GLub} proves Proposition \ref{prop-ub}. By definition of the minimizer
$(\psi,\Ab)$, we have,
$$\mathcal E(\psi,\Ab)\leq \mathcal E(\varphi,\Ab;\Omega)\,.$$
Since $\mathcal E(\psi,\Ab;\Omega)=\mathcal
E(\psi,\Ab;\Omega\setminus Q_{\ell}(x_0))+\mathcal
E_0(\psi,\Ab;Q_{\ell}(x_0))$, the estimate \eqref{eq-GLub} gives us,\\
$$\mathcal  E_0(\psi,\Ab;Q_{\ell}(x_0))\leq \frac{(1+\delta)}{b}
m_0(b,R)+ r_0(\kappa)\,.$$ Dividing both sides by
$|Q_{\ell}(x_0)|=\ell^2$ and remembering the definition of $r_0(\kappa)$,
we get,
\begin{equation}\label{eq-ub-final}\frac1{|Q_{\ell}(x_0)|}\mathcal E_0(\psi,\Ab,Q_{\ell}(x_0))\leq
\displaystyle\frac{(1+\delta)}{b\ell^2}m_0(b,R) +C
\left(\delta+\delta^{-1}\ell^4\kappa^{2}+\frac1{\ell\kappa\sqrt{\epsilon}}+\delta^{-1}\ell^{2\alpha}\right)\kappa^2\,.
\end{equation}
The inequality in \eqref{eq-remainder} tell us that $m_0(b,R)\leq R^{2}g(b)+CR$
for all $b\in[0,1]$ and $R$ sufficiently large. We substitute this
into \eqref{eq-ub-final} and we select $\delta=\ell$, so that
$$r_0(\kappa)=\kappa^2\mathcal
O\left((\ell\kappa\sqrt{\epsilon})^{-1}+\ell^{3}\kappa^{2}+\ell^{2\alpha-1}\right)\,.$$
Using \eqref{defbR} we get,
\begin{align}
\frac1{|Q_{\ell}(x_0)|}\mathcal E(\psi,\Ab,Q_{\ell}(x_0))&\leq
\frac{(1+\delta)R^{2}}{b\ell^{2}}g(b)+\frac{CR}{b\ell^{2}}+\kappa^{2}\mathcal
O\left((\ell\kappa\sqrt{\epsilon})^{-1}+\ell^{3}\kappa^{2}+\ell^{2\alpha-1}\right)\nonumber\\
&\leq
g\left(\displaystyle\frac{H}{\kappa}\underline{B}_{Q_{\ell}(x_0)}\right)\kappa^2+C\left((\ell\kappa\sqrt{\epsilon})^{-1}+\ell^{3}\kappa^{2}+\ell^{2\alpha-1}\right)\kappa^2\,.\nonumber
\end{align}
This establishes the result of Proposition~\ref{prop-ub}.
\end{proof}
\subsection{Proof of Theorem~\ref{thm-2D-mainloc}, upper bound}
The parameters $\ell$ and $\epsilon$ have the same form as in \eqref{defellepsilon} and we take the same choice of $\beta$ and $\mu$ as in \eqref{beta-mu}. Consider the  lattice $\Gamma_{\ell} :=\ell\Z\times\ell\Z$ and  write, for $\gamma \in \Gamma_{\ell}$, $Q_{\gamma,\ell}=Q_{\ell}(\gamma)$. For any $\gamma \in \Gamma_{\ell}$ such that $\overline{Q_{\ell}(\gamma)}\subset\Omega\cap\{|B_0|>\epsilon\}$, let:
$$ \mathcal
I_{\ell,\epsilon}(D)=\{\gamma~:~\overline{Q_{\gamma,\ell}}\subset D\cap\{|B_0|>\epsilon\}\}\,,\,\,\,\,\,\,\,
N={\rm Card}\,\mathcal I_{\ell,\epsilon}(D), $$
and
$$D_{\ell,\epsilon}=\text{int}\left(\displaystyle{\cup_{\gamma\in \mathcal
I_{\ell,\epsilon}(D)}}\overline{Q_{\gamma,\ell}}\right)\,.
$$
Notice that, by \eqref{B(x)},
$$ N=|D|\ell^{-2}+{\mathcal O}(\epsilon\ell^{-2}) +{\mathcal O}(\ell^{-1})\mbox{ as } \ell\to 0~\mbox{and}~\epsilon\to 0\,.
$$
If $(\psi,\Ab)$ is a minimizer of \eqref{eq-2D-GLf}, we have,
\begin{equation}\label{eq-loc-en}
\mathcal E(\psi,\Ab;D)=\mathcal E_0(\psi,\Ab;D_{\ell,\epsilon})+\mathcal
E_0(\psi,\Ab;D\setminus D_{\ell,\epsilon}) + (\kappa H)^2
\int_{\Omega} |\curl\big(\Ab - \Fb\big)|^2\,dx \,.
\end{equation}
Using Corollary~\ref{corol-2D-main}, we may write,
\begin{equation}\label{eq-loc-en1}
\mathcal E(\psi,\Ab;D)\leq\mathcal E_0(\psi,\Ab;D_{\ell,\epsilon})+\mathcal E_0(\psi,\Ab;D\setminus D_{\ell,\epsilon})+C\kappa^{\tau_0}\,.
\end{equation}
Here $\tau_0\in(1,2)$. Notice that
\begin{equation}\label{D-D_ell}
|D\setminus D_{\ell,\epsilon}|=\mathcal{O}\left(\ell|\partial D_{\ell,\epsilon}|+\epsilon\right)\,.
\end{equation}
We get by using \eqref{eq-psi<1} and
\eqref{eq-est1}:
\begin{align}
 \left|\mathcal E_0(\psi,\Ab;D\setminus D_{\ell,\epsilon})\right|&\leq|D\setminus D_{\ell,\epsilon}|\left(C_1\kappa^{2}\|\psi\|_{L^{\infty}(D)}^{2}+\kappa^{2}\|\psi\|_{L^{\infty}(D)}^{2}+\frac{\kappa^{2}}{2}\|\psi\|_{L^{\infty}(D)}^{4}\right)\nonumber\\
&\leq C_2(\ell+\epsilon)\kappa^2\,\label{loc-eq-en-bnd}.
\end{align}
To estimate $\mathcal E_0(\psi,\Ab;D_{\ell,\epsilon})$, we notice that,
$$\mathcal E_0(\psi,\Ab;D_{\ell,\epsilon})=\sum_{\gamma\in\mathcal I_{\ell,\epsilon}(D)}\mathcal E_0(\psi,\Ab;Q_{\gamma,\ell})\,.$$
Using Proposition~\ref{prop-ub} and the estimates in \eqref{loc-eq-en-bnd} with $\beta=\frac{3}{4}$, $\alpha=\frac{2}{3}$ and $\mu=\frac{1}{8}$, we get,
\begin{align}
\mathcal E_0(\psi,\Ab;D)&\leq\sum_{\gamma\in\mathcal
I_{\ell,\epsilon}(D)}g\left(\frac{H}{\kappa}\underline{B}_{Q_{\ell}(x_0)}\right)\kappa^2\ell^2+C\left(\ell^{3}\kappa^{2}+\ell^{2\alpha-1}+(\ell\kappa\sqrt{\epsilon})^{-1}+\epsilon\right)\kappa^2+C_{1}\kappa^{\tau_{0}}\nonumber\\
&\leq\kappa^{2}\sum_{\gamma\in\mathcal
I_{\ell,\epsilon}(D)}g\left(\frac{H}{\kappa}\underline{B}_{Q_{\ell}(x_0)}\right)\ell^2+C_{2}\kappa^{\tau_{0}}\nonumber\,,
\end{align}
where $$\underline{B}_{Q_{\ell}(x_0)}=\displaystyle\sup_{x\in Q_{\ell}(x_0)} B_{0}(x)\,.$$
Recognizing the lower Riemann sum of $x\mapsto g\left(\displaystyle \frac{H}{\kappa}B_{0}(x)\right)$, and using the monotonicity of $g$ we get:

\begin{equation}\label{E0<g-D}
\mathcal E_0(\psi,\Ab;D)\leq\kappa^{2}\int_{D_{\ell,\epsilon}}g\left(\displaystyle \frac{H}{\kappa}B_{0}(x)\right)\,dx+C_{2}\kappa^{\tau_{0}}\,.
\end{equation}
Thus, we get by using \eqref{D-D_ell} and the property of $g$ in Theorem~\ref{thm-thmd-AS},
$$\kappa^{2}\int_{D_{\ell,\epsilon}}g\left(\displaystyle \frac{H}{\kappa}B_{0}(x)\right)\,dx\leq\kappa^{2}\int_{D}g\left(\displaystyle \frac{H}{\kappa}B_{0}(x)\right)\,dx +C_{3}\kappa^{\tau_{0}}\,.$$
This finishes the proof of the upper bound.

\subsection{Lower bound} We keep the same notation as in the derivation of the upper bound. We start with \eqref{eq-loc-en} and write,
\begin{equation}\label{eq-loc-en2}
\mathcal E(\psi,\Ab;D)\geq\mathcal E_0(\psi,\Ab;D_{\ell,\epsilon})+\mathcal E_0(\psi,\Ab;D\setminus D_{\ell,\epsilon})\,.
\end{equation}
Similarly, as we did for the Lower bound~\ref{lower bound}, we get,
\begin{equation}\label{up-E-D}
\mathcal
E_0(\psi,\Ab;D)\geq\kappa^{2}\int_{D}g\left(\frac{H}{\kappa}B_{0}(x)\right)\,dx-C\kappa^{\tau_{0}}\,.
\end{equation}
This finish the proof of Theorem~\ref{thm-2D-mainloc}.

\section{Proof of  Theorem~\ref{thm-2D-op}}
\subsection{\textbf{Proof of\eqref{eq-2D-op'}}}
Let $(\psi,\Ab)$ be a solution of \eqref{eq-2D-GLeq} and $\tau_1=\tau_{0}-2$. Then $\psi$
satisfies,
\begin{equation}\label{eq-GLeq-1}
-(\nabla-i\kappa H\Ab)^2\psi=\kappa^2(1-|\psi|^2)\psi\,\quad{\rm
in}\quad\Omega\,.
\end{equation}
We multiply both sides of the equation in \eqref{eq-GLeq-1} by
$\overline{\psi}$ then we integrate over $D$. An integration
by parts gives us,
\begin{equation}\label{1eq}
\int_{D}\left(|(\nabla-i\kappa H\Ab)\psi|^2-\kappa^2|\psi|^2+\kappa^2|\psi|^4\right)\,dx
-\int_{\partial D}\nu\cdot(\nabla-i\kappa
H\Ab)\psi\,\overline{\psi}\,d\sigma(x)=0\,.
\end{equation}
Using the estimates \eqref{eq-psi<1}, \eqref{kappaH} and \eqref{eq-est1} , we get
that the boundary term which is not necessary 0 if $D\neq\Omega$ above is $\mathcal O(\kappa)$\,.
So, we rewrite \eqref{1eq} as follows,
\begin{equation}\label{eq-proof-op}-\frac12\kappa^2\int_{D}|\psi|^4\,dx=\mathcal E_0(\psi,\Ab;D)+\mathcal O(\kappa)\,.\end{equation}
Using \eqref{up-E-D}, we conclude that,
\begin{equation}\label{eq-op-ub}
\frac12\int_{D}|\psi|^4\,dx\leq
-\int_{D}g\left(\frac{H}{\kappa}B_{0}(x)\right)dx + C\kappa^{\tau_1}.
\end{equation}
\subsection{\textbf{Proof of \eqref{eq-2D-op}}}
If $(\psi,\Ab)$ is a minimizer of \eqref{eq-2D-GLf},  then
\eqref{eq-proof-op} is still true. We apply in this case
Theorem~\ref{thm-2D-mainloc} to write an upper bound of $\mathcal
E_0(\psi,\Ab;D)$. Consequently, we deduce that,
\begin{equation}\label{eq-op-lb}
\frac12\int_{D}|\psi|^4\,dx\geq -\int_{D}g\left(\frac{H}{\kappa}B_{0}(x)\right)dx-C\kappa^{\tau_1}.
\end{equation}
Combining the upper bound in \eqref{eq-op-lb} with the lower bound
in \eqref{eq-op-ub} finishes the proof of Theorem~\ref{thm-2D-op}.

\appendix
\section{}~\\
\subsection{\textbf{$L^{p}$-regularity for the curl-div system}}~\\
We consider the two dimensional case. We denote, for $k\in \N$, by $W^{k, p}_{{\rm div}}(\Omega)$ the space
 $$W^{k,p}_{{\rm div}}(\Omega)=\{\Ab\in W^{k,p}(\Omega), {\rm div}\Ab=0\,{\rm and}\,\Ab\cdot\nu=0\,{\rm on}\,\partial\Omega\}.$$
Then we have the following $L^p$ regularity for the curl-div system.
\begin{prop}\label{curl-div-reg}
Let $1\leq p<\infty$. If $\Ab\in W^{1,p}_{\rm div}(\Omega)$ satisfies $\curl \Ab\in W^{k,p}(\Omega)$, for some $k\geq 0$, then $\Ab\in W^{k+1,p}_{\rm div}(\Omega)$.
\end{prop}
\begin{proof}
If $\Ab$ belongs to $W^{1,p}_{\rm div}(\Omega)$ and $\curl \Ab\in L^{p}(\Omega)$, then there exists $\psi\in W^{2,p}(\Omega)$ such that
$\Ab=(-\partial _{x_{2}}\psi,\partial_{x_{1}}\psi)$, $-\Delta \psi=\curl\Ab$, with $\psi=0 $ on $\partial\Omega$.
This is simply the Dirichlet $L^p$ problem for the Laplacian (See \cite{FH-p}, Section A.1). The result we need for proving the proposition is then that if $-\Delta \psi$ is in addition in $W^{k,p}(\Omega)$ then $\psi\in W^{k+2,p}(\Omega)$. This is simply an $L^p$ regularity result for the Dirichlet problem for the Laplacian which is described in (\cite{FH-p}, Section F.4).
\end{proof}
\subsection{\textbf{Construction of $\varphi_{x_0}$.}}
\begin{lem}\label{curlF0}
If $B_{0}\in L^{2}(\Omega)$, then there exists a unique $\Fb\in H^{1}_{\Div}(\Omega)$ such that,
\begin{equation}\label{A_F}
\curl\Fb=B_{0}\,.
\end{equation}
\end{lem}
\begin{proof}
The proof is standard, see \cite{GR}. Let
$\Fb=\left[\begin{array}{c}
\partial_{x_{2}}f\\
-\partial_{x_{1}}f\end{array}\right]\,,$ where $f\in
H^{2}(\Omega)\bigcap H^{1}_{0}(\Omega)$ is the unique solution of
\begin{equation}-\Delta f=B_{0}\quad{\rm
in~}\Omega \,.
\end{equation}
Then we deduce from the Dirichlet condition satisfied by $f$  that $\tau\cdot \nabla f=0\quad{\rm
on~}\partial\Omega $
which  is equivalent to $\nu\cdot\Fb=0\quad{\rm
on~}\partial\Omega\,.$  This finishes the proof of Lemma
\ref{curlF0}.
\end{proof}
We continue with a lemma that will be useful in estimating the Ginzburg-Landau functional.
\begin{lem}\label{app F}
There exists a positive constant $C$ such that,  if  $\ell\in (0, 1)$ and $x_{0} \in\Omega $ are such that $\overline{Q_{\ell}(x_{0})} \subset \Omega$,
then for any $\widetilde{x_{0}}\in\overline{Q_{\ell}(x_{0})}$, there exists a function $\varphi_{0}\in
H^{1}(\Omega)$ such that the magnetic potential
$\Fb$ satisfies,
\begin{equation}{\label{lem-F}}
|\Fb(x)-\nabla\varphi_{0}(x)-B_{0}(\widetilde{x_{0}})\Ab_{0}(x-x_{0})|\leq
C\ell^{2}\,,\,\,\,\,\,\,\,\Big(x\in Q_{\ell}(x_{0})\Big)\,,
\end{equation}
where $B_{0}$ is the function introduced in \eqref{B(x)} and $\Ab_0$ is the magnetic potential
introduced in \eqref{eq-hc2-mpA0}.
\end{lem}
\begin{proof}
We use $\mathrm{T}$aylor formula near $\widetilde{x_{0}}$ to order
$2$ and get:
\begin{equation}\label{second appF}
\Fb(x)=\Fb(\widetilde{x_{0}})+M(x-\widetilde{x_{0}})+\mathcal
O(|x-\widetilde{x_{0}}|^2)\,,\,\,\,\,\,\,\,\,\, \forall x\in
Q_{\ell}(x_0)\,,
\end{equation}
 where $$M=D\Fb(\widetilde{x_{0}})=
\left[\begin{array}{cc} \displaystyle\frac{\partial\Fb^{1}}{\partial x_{1}}{\displaystyle\mid_{\widetilde{x_{0}}}}&\displaystyle\frac{\partial\Fb^{1}}{\partial x_{2}}{\displaystyle\mid_{\widetilde{x_{0}}}}\\
\displaystyle\frac{\partial\Fb^{2}}{\partial
x_1}{\displaystyle\mid_{\widetilde{x_{0}}}}&\displaystyle\frac{\partial\Fb^{2}}{\partial
x_2}{\displaystyle\mid_{\widetilde{x_{0}}}}
\end{array}
\right]\,.$$ We can write $M$ as the sum of two matrices,
$M=M^{s}+M^{as}$, where $M^{s}=\frac{M+M^{t}}{2}$ is symmetric and $M^{as}=\frac{M-M^{t}}{2}$ is antisymmetric.\\
Notice that $\curl
\Fb(\widetilde{x_{0}})=\displaystyle\frac{\partial \Fb^{2}}{\partial
x_1}{\displaystyle\mid_{\widetilde{x_{0}}}}-\displaystyle\frac{\partial
\Fb^{1}}{\partial
x_2}{\displaystyle\mid_{\widetilde{x_{0}}}}=B_{0}(\widetilde{x_{0}})$.
Consequently, $$M^{as}=\left[\begin{array}{cc}
0&-B_{0}/2\\
B_{0}/2&0
\end{array}
\right]\,.$$
Substitution into $M$ gives as that,
$$M(x-x_{0})=\nabla\phi_{0}(x)+B_{0}(\widetilde{x_{0}})\Ab_{0}(x-x_{0})\,,$$
where $\Ab_{0}(x)=\frac{1}{2}(-x_2,x_1)$ and the function
$\phi_{0}$ is defined by
$$\phi_{0}(x)=\frac{1}{2}\left\langle\left(\frac{M+M^{t}}{2}\right)(x-x_{0}),(x-x_{0})\right\rangle\,.$$
Let
$\varphi_{0}(x)=\phi_{0}(x)+\left(\Fb(\widetilde{x_{0}})+M(x_{0}-\widetilde{x_{0}})\right)\cdot
x\,.$ Substitution into \eqref{second appF} gives as that,
$$\Fb=B_{0}(\widetilde{x_{0}})\Ab_{0}(x-x_{0})+\nabla\varphi_{0}(x)+\mathcal{O}\left(|x-\widetilde{x_{0}}|^{2}\right)\,.$$
Notice that, if $x\in Q_{\ell}(x_0)\,,$ then
$|x-\widetilde{x_{0}}|\leq\ell\sqrt{2}$. This finishes the proof of Lemma
\ref{app F}.
\end{proof}

\begin{rem}
We will apply this lemma by considering $\widetilde x_0$ such that $B_0(\widetilde x_0) = \sup_{Q_\ell(x_0)} B_0(x)$ or
$B_0(\widetilde x_0) = \inf_{Q_\ell(x_0)} B_0(x)$\,.
\end{rem}
\section*{Acknowledgements}
This work is partially supported by a grant from Lebanese University.
I would like to thank my supervisors \textit{A.Kachmar} and \textit{B.Helfer} for their support.
%
%
%
%

\end{document}